\documentclass[reqno,12pt]{article}
\usepackage{amsmath, amsthm, amssymb, euscript}

\usepackage{graphicx}
\usepackage{pstricks} 

\usepackage{color}

\newtheorem{theorem}{Theorem}[section]
\newtheorem{remark}{Remark}[section]
\newtheorem{corollary}[theorem]{Corollary}
\newtheorem{lemma}[theorem]{Lemma}

\newtheorem{definition}{Definition}[section]

\numberwithin{equation}{section}

\def\al{\alpha}

\def\de{\delta}
\def\ep{\epsilon}

\def\e {\varepsilon}

\newcommand{\vt}{\overrightarrow}
\newcommand{\R}{\mathbb{R}}

\newcommand{\bP}{\mathbf{P}}
\newcommand{\bL}{\mathbf{L}}
\newcommand{\bV}{\mathbf{V}}

\newcommand{\B}{\mathbf{B}}
\newcommand{\Q}{\mathcal{Q}}
\newcommand{\bO}{\mathbf{O}}

\newcommand{\E}{\mathbf{E}}
\newcommand{\G}{\mathbf{G}}

\renewcommand{\Re}{\mathop{\mathrm{Re}}}
\renewcommand{\Im}{\mathop{\mathrm{Im}}}

\newcommand{\norm}[1]{\left\Vert #1 \right\Vert}
\newcommand{\wei}[1]{\langle #1 \rangle}

\newcommand{\bke}[1]{\left( #1 \right)}

\newcommand{\Pc}{\, \mathbf{P}\! _\mathrm{c} \, \! }

\newcommand{\Pe}{\, \mathbf{P}\! _\mathrm{1} \, \! }

\renewcommand{\L}{\mathcal{L}}

\newcommand{\bK}{\mathbf{K}}
\newcommand{\bH}{\mathbf{H}}
\newcommand{\bU}{\mathbf{U}}

\newcommand{\loc}{_{\mathrm{loc}}}
\newcommand{\la}{\lambda}

\newcommand{\wt}{\widetilde}
\newcommand{\conj}{\mathop{\,\mathbf{conj}\,}}  

\newcommand{\donothing}[1]{}

\newcommand{\les}{\lesssim}

\begin{document}

\title{Stable Directions for Degenerate Excited States of Nonlinear 
Schr\"odinger Equations}

\author{Stephen Gustafson, \quad Tuoc Van Phan}

\maketitle

\abstract{
We consider nonlinear Schr\"{o}dinger equations, 
$i\partial_t \psi = H_0 \psi + \lambda |\psi|^2\psi$ in 
$\mathbb{R}^3 \times [0,\infty)$, where $H_0 = -\Delta + V$, $\lambda=\pm 1$, 
the potential $V$ is radial and spatially decaying, and the linear
Hamiltonian $H_0$ has only two eigenvalues $e_0 < e_1 <0$, where $e_0$
is simple, and $e_1$ has multiplicity three. We show that there exist
two branches of small ``nonlinear excited state'' standing-wave
solutions, and in both the resonant ($e_0 < 2e_1$) and non-resonant 
($e_0 > 2e_1$) cases, we construct certain finite-codimension
regions of the phase space consisting of solutions converging 
to these excited states at time infinity (``stable directions'').
}\ \\



\section{Introduction}
We consider the nonlinear Schr\"{o}dinger equation
\begin{equation} \label{NSE}
  i\partial_t \psi = H_0 \psi + \lambda |\psi|^2 \psi
\end{equation}
which arises in several physical settings, including 
many-body quantum systems, and optics. Here
the wave function $\psi = \psi(x,t)$ is complex-valued,
\[
  \psi : \mathbb{R}^3 \times [0,\infty) \to \mathbb{C},
\]
and the linear Hamiltonian is
\[ 
  H_0 := -\Delta + V(x),
\]
where $V : \R^3 \to \R$ is a smooth, spatially decaying potential 
function. We take $\lambda = \pm 1$.  
Equation \eqref{NSE} can be expressed as a Hamiltonian system
\[ 
  i\partial_t \psi = \frac{ \partial \mathcal{E}[\psi,\bar{\psi}]}
  {\partial \bar{\psi}},
\]
where the Hamiltonian energy is defined as
\[ 
  \mathcal{E}[\psi] = \mathcal{E}[\psi, \bar{\psi}] = 
  \int \left (\frac{1}{2} \nabla \psi \cdot \nabla \bar{\psi} 
  + \frac{1}{2} V(x) \psi \bar{\psi} +\frac{\lambda}{4}\psi^2
  \bar{\psi}^2 \right) dx.
\]
Since the energy is invariant under the time-translation 
$t \mapsto t + t_0$ ($t_0 \in \R$) and the phase rotation 
\[
  \psi \mapsto e^{i r} \psi,
  \quad \quad  r \in \mathbb{R},
\] 
the energy and the particle number 
\[
  \mathcal{N}[\psi] := \int |\psi(x)|^2dx
\]
are constant in time for any smooth solution 
$\psi(t) = \psi(.,t) \in H^1(\mathbb{R}^3)$ of \eqref{NSE}:
\[ 
  \mathcal{E}[\psi(t)] = \mathcal{E}[\psi(0)], \quad 
  \mathcal{N}[\psi(t)] = \mathcal{N}[\psi(0)]
\]
The global well-posedness of solutions with $\| \psi(0) \|_{H^1}$ small
can be established easily by using these conserved quantities, and a
continuity  argument. 

A very important feature of the equation \eqref{NSE} is 
that it can have localized ``standing wave'' or ``nonlinear bound state'' 
solutions. These are solutions of the form $\psi(x,t) = e^{-iEt}Q(x)$, 
where the profile function $Q$ must therefore solve the equation
\begin{equation} \label{Q.eqn}
  H_0 Q + \lambda |Q|^2 Q = EQ.
\end{equation}
The solutions of \eqref{Q.eqn} are critical points of the Hamiltonian 
$\mathcal{E}$ subject to the constraint that the $L^2$-norm of $Q$ is fixed. 
We may obtain {\it small} solutions of \eqref{Q.eqn} as bifurcations along 
the discrete eigenvalues of $H_0$. In this paper, we assume that $H_0$
has two discrete eigenvalues $e_0 < e_1 < 0$. The small solutions with $E$ close to 
$e_0$ are called \textit{nonlinear ground states}, while those with $E$ close 
to $e_1$ are naturally called \textit{nonlinear excited states}. 

Along a simple eigenvalue, the bifurcation problem for finding the 
corresponding nonlinear solutions (eg. nonlinear ground states) is quite 
standard, regardless of the multiplicities of the other eigenvalues, see 
\cite{SW1, TY1, GNT}. Along a degenerate eigenvalue (such as $e_1$
in our setting), however, the problem becomes 
more delicate due to the interaction between the different directions in the 
$e_1$-eigenspace. Hence our first goal in this paper is to find all of
the nonlinear excited states when $e_1$ is degenerate. 

We consider here a {\it radial} potential $V$, and we assume that 
$e_1$ has multiplicity three, corresponding to the first non-zero
angular momentum spherical harmonics:

\medskip

\noindent 
{\bf Assumption A1:} 
\textit{The potential $V(x)$ is spherically symmetric, 
  and the linear Hamiltonian $H_0 = -\Delta + V$ has
  a discrete eigenvalue $e_1 < 0$ of multiplicity three, 
  with eigenspace 
\[
  \bV := Null(H_0 - e_1) = \mbox{ span } 
  \{\phi_1, \phi_2, \phi_2\}, \quad\quad
  \phi_j(x) = x_j \varphi(|x|)
\]
for a real function $\varphi$.}

\noindent
Notice that such a multiplicity-three eigenvalue is a 
generic possibility in three dimensions with a radial
potential.

\medskip

We will denote the orthogonal projection onto this 
excited-state eigenspace by
\[
  \bP_1 = \bP_1(H_0) := \mbox{ orthogonal projection onto }
  \bV.
\]

Since the potential $V$ is radial, equation~\eqref{Q.eqn} is 
invariant not only under phase rotation, time translation, and 
complex conjugation, but also under all spatial rotations 
and reflections:
\[
  Q(x) \mapsto Q(\Gamma x), \quad \quad
  \Gamma \in O(3).
\]  
This rich structure plays a central role in 
understanding the existence of nonlinear excited states, as well 
as the nearby dynamics.

In Section~\ref{Exi}, we shall see that there is a \textit{symmetry
  breaking bifurcation phenomenon} in the bifurcation equations of the
nonlinear excited states $Q$. Here, we shall briefly describe the
phenomenon (for the systematic theory, see~\cite{GSS, GS, SA}). 
By integrating \eqref{Q.eqn} against a linear excited state 
$\phi_j$, it follows that there are non-trivial solutions close 
to $\bV$ only if $\lambda(E-e_1) > 0$. 
We write $Q = \ep (v+h)$ for $v \in \bV$, 
$h \in \bV^\perp$ and $\ep^2 = [E-e_1]/\lambda$. Then, the equation 
\eqref{Q.eqn} is equivalent to 
\begin{equation} \label{Q.bifur}
\left \{
  \begin{array}{ll}
  & h = \lambda \epsilon^2 [H_0-e_1]^{-1}(1- \Pe(H_0)) 
  \{h - |v +h|^2(v +h) \} \\ 
  & \Pe(H_0) [v - |v+h|^2(v +h)] = 0 
  \end{array} \right..
\end{equation}
By using the contraction 
mapping theorem, we obtain the unique small solution $h = h(\ep,v)$ 
of the first equation of \eqref{Q.bifur} for sufficiently small 
$\ep>0$. To solve the bifurcation equation 
\[
  N(\ep,v) : =\Pe(H_0) [v - |v+h(\ep,v)|^2(v +h(\ep,v))] = 0, 
\]
the standard method is to apply the implicit function theorem. 
However,
due to the invariant structure of $N$, its derivative $N_v(0,v)$
has a non-trivial kernel for all $v \in \bV$ such that $N(0,v)=0$. 
To overcome this, we have to restrict $N$ onto $N$-invariant
subspaces of $\bV$, and so eliminate the kernel of the 
derivative of $N$. 
In this way, we obtain two particular representatives,
denoted  $Q_E$ and $\wt{Q}_E$, of the family of nonlinear excited 
states. Then, applying the symmetries -- spatial rotation, 
reflection, phase rotation and complex conjugation -- to $Q_E$ 
and $\wt{Q}_E$, we obtain all of the nonlinear excited states of 
\eqref{Q.eqn}, as summarized in the following 
theorem, proved in Section~\ref{Exi}:

\begin{theorem}[Existence of degenerate excited states] 
\label{ex.sol}
There exist two branches of nonlinear excited state solutions to
\eqref{Q.eqn}, with $E - e_1$ small. These branches
are generated by applying phase rotations 
($Q(x) \mapsto e^{i r} Q(x), \; r \in \R$) and spatial rotations
($Q(x) \mapsto Q(\Gamma x), \; \Gamma \in SO(3)$) 
to two fixed solutions of the forms
\[ 
  Q_E(x) = x_1f_1(x_1^2, x_2^2+x_3^2), \quad 
  \wt{Q}_E(x) = e^{i\theta}f_2(x_1^2 +x_2^2, x_3^2)
\]
for real-valued functions $f_1, f_2$, and where
$\theta \in [0,2\pi)$ is the angle between the $x_1$ axis and the 
vector $x' = (x_1,x_2) \in \mathbb{R}^2$. The solutions
$Q_E$ and $\wt{Q}_E$ lie in the Sobolev space $H^2$,
and decay exponentially as $|x| \to \infty$.
\end{theorem}
\begin{remark}
Another way to state this result: 
the excited state branches arise as the orbits, under the symmetry 
group of the equation, of two particular bifurcation curves 
of solutions. One of these ($Q_E$) is real, odd in one direction, 
and invariant under rotations fixing that direction, while the 
other ($\wt{Q}_E$) is even in one direction, and ``co-rotational''
with respect to rotations fixing that direction.
\end{remark}
\begin{remark}
The nature of the symmetry group, together with the symmetry 
properties of $Q_E$ and $\wt{Q}_E$, imply that each solution
branch is a four-dimensional family: one degree of
freedom is the eigenvalue $E$, one is the phase rotation,
and two come from spatial rotations. See~\eqref{parameters} 
for an explicit expression of the branches in terms of parameters.
\end{remark}

Our next goal is to study the dynamics nearby the nonlinear excited 
states. It is well-known that the family of {\it ground states} -- let
us denote them by $Q_E^0$, with nonlinear eigenvalue $E$ close to
$e_0$ -- is \textit{orbitally stable}; that is the orbit of
a ground state under the action of the symmetry group 
(in this case just phase rotations, since $Q_E^0$ is
radially symmetric) is stable: 
\[
  ||\psi(0) - Q_E^0||_{H^1} \mbox{ small } \implies
  \inf_{r} \| \psi(t) - Q_E^0 e^{i r} \|_{H^1}
  \mbox{ small for } t \geq 0
\] 
(see \cite{RW}). One may expect that $Q_{E}^0$ is even 
{\it asymptotically stable} in, say, a local $L^2$-norm:
\begin{equation} \label{asy.st}  
  \lim_{t \rightarrow \infty} \inf_{E,r}
  ||\psi(t) - Q_{E}^0 e^{i r}||_{L^2_{loc}} = 0 
\end{equation}
where one now has to consider the entire ground state family;
i.e., to allow for modulation of the eigenvalue $E$ as well as 
the phase $r$. If $H_0$ has only one discrete eigenvalue, 
it is proved in~\cite{SW1} that solutions initially close to a 
ground state eventually converge to some (other) ground state 
$e^{i r} Q_{E}^0$. In the case $H_0$ has two discrete 
eigenvalues $e_0 < e_1 < 0$ (also satisfying a resonance
condition) and both of them are simple, this is proved 
in \cite{TY1}. Recently, if $e_1$ is degenerate, the same 
result is proved in \cite{ZW}, by introducing a new type of 
normal form. 

If the initial data is not close to a ground state, the presence of
excited states makes the problem more subtle, -- see \cite{NTP, Tsai,
  TY2}. The physical intuition is that excited states are unstable, 
and that, generically, nearby states should radiate, and then relax 
(locally) to the ground state. However, the results of \cite{TY3} show 
that there are, in fact, {\it stable directions} -- finite 
co-dimension families of solutions which converge to the excited 
states -- at least when $H_0$ has simple eigenvalues. 
Motivated by the papers \cite{ZW} and \cite{TY3}, in this paper 
we shall construct solutions converging to the nonlinear excited 
states of~Theorem \ref{ex.sol}. Of course, this result does not 
contradict the physical intuition, since this family of solutions 
has zero measure in some sense (and so should not be directly seen 
in experiments or numerical simulations). 

To study the stable directions for the nonlinear excited states, 
it is essential that we understand the
linearized operators around these excited states, their spectral
properties, and the associated time-decay estimates. 
In particular, if we denote the {\it linearized operator} around 
a solitary wave solution $e^{i E t} Q(x)$ by
$\L_Q$, so that
\[
  \L_Q \zeta = -i \{(H_0 -E + 2\lambda |Q|^2) \zeta + 
  \lambda Q^2 \bar{\zeta} \},    
\]
then the analysis of Section~\ref{sp.an} shows that
for each of the nonlinear excited states, there is a 
finite-codimension subspace 
\[
  \E_c(\L_{Q}) = 
  \mbox{ the continuous spectral subspace for }
  \L_{Q} 
\]
on which hold the dispersive decay estimates
\[ 
  \norm{e^{t \L_{Q}} \eta}_{L^p} \lesssim  
  |t|^{-3(1/2-1/p)}\norm{\eta}_{L^{p'}},   
  \quad \frac{1}{p} + \frac{1}{p'} = 1. 
\]
It turns out that the spectrum of the linearized operator around 
$Q_E$ (and its symmetry translates) is different from that of the 
linearized operator around $\wt{Q}_E$ (and its symmetry
translates). This 
interesting phenomena is due to the difference in the symmetry 
properties of $Q_E$ and $\wt{Q}_E$. Moreover, because of the 
degeneracy, there is an interaction between many different 
directions in the same modes, and it is more complicated to study 
these linearized operators, and to construct the stable directions,
than it is for the ground state.

Before stating our main theorem, we make precise our further 
assumptions on the potential $V$.

\medskip

\noindent {\bf Assumption A2:} 
\textit{
The linear Hamiltonian $H_0 = -\Delta + V$ has only two
eigenvalues $e_0 < e_1 < 0$, with $e_0 \not= 2 e_1$.
}

\medskip

\begin{remark}
Taken together, Assumptions A1 and A2 say that: the radial 
eigenfunction problem supports only the eigenvalue $e_0$;
the eigenvalue problem corresponding to the first 
non-zero angular momentum sector supports only the eigenvalue 
$e_1$; and there are no eigenvalues corresponding to higher
angular momenta. It is not difficult to construct examples
of such potentials (eg. among finite square well
potentials of varying depth and width).
\end{remark}

Next we make precise our assumptions of spatial decay and 
regularity of the potential $V(x)$.

\medskip

\noindent {\bf Assumption A3:} 
\textit{
For some $\sigma>0$,
\[ 
  |\nabla^\alpha V(x)| \leq C (1 +|x|^2)^{-5-\sigma}, \ 
  \forall x \in \mathbb{R}^3, \ |\alpha| \leq 2, 
\]
and there is $0<\sigma_0 <1$ such that
\[ 
  \norm{[(x \cdot \nabla)^k V]\phi}_{L^2} \leq 
  \sigma_0 \norm{-\Delta \phi}_{L^2} + C \norm{\phi}_{L^2}, \ 
  \forall \ k = 1,2,3, \ \phi \in H^1.
\]
Furthermore, $0$ (the bottom of the continuous
spectrum of $H_0$) is not an eigenvalue nor a resonance.}

\medskip

Assumption A3 ensures we can apply standard analysis tools 
for linear Schr\"{o}dinger operators. It is certainly not optimal.

The final assumption will ensure, in the {\it resonant case}
$e_0 < 2 e_1$, that the resonant interaction is generic.
Denote a normalized ground-state eigenfunction by $\phi_0$,
which we may suppose is positive and radial (since $V$ is radial),
denote by $\bP_0(H_0)$ the corresponding orthogonal projection,
\[
  Null(H_0 - e_0) = \mathbb{C} \phi_0, \quad\quad 
  \phi_0(x) = \phi_0(|x|) > 0, \quad\quad
  \bP_0 = \bP_0(H_0) = \langle \; \phi_0 \; | \;\; ,
\] 
and denote by $\Pc$ the orthogonal projection onto the 
continuous spectral subspace of $H_0$:
\[
  \Pc = \Pc(H_0) = 1 - \bP_0(H_0) - \bP_1(H_0).
\]

\medskip 

\noindent {\bf Assumption A4:}
\textit{``Fermi Golden Rule'':  
If $e_0 < 2 e_1$ (``resonant case''), there exists
$\lambda_0 >0$ such that
\begin{equation} \label{FGR}
\lim_{r \to 0+} \Im \bke{\phi_{0}\phi^2 , (H_0 - 2e_1 + e_0
  -ir)^{-1}\Pc \phi_{0} \phi^2 } \geq \lambda_0 \norm{\phi}_{L^2}^4 
\quad\quad
\forall \phi \in \bV .
\end{equation} 
}

\medskip

When $e_1$ is simple, the condition \eqref{FGR} is
well-known and appears in many papers, 
see \cite{BP1, BP2, BS, ZS1, NTP, Tsai, TY1, TY2, TY3, TY4}. 
In the degenerate case, \eqref{FGR}
is also used in \cite{ZW}, and it is claimed there that 
\eqref{FGR} holds generically. 

The existence of stable directions for the excited 
states is proved in Section~\ref{prom}:

\begin{theorem}[Stable directions for degenerate excited states] 
\label{m-theorem} 
Assume that $H_0 = -\Delta + V$ satisfies A1-A4 above. 
Then, there exists $\ep_0>0$ such that for any 
$0 < \ep \leq \ep_0$, 
there is $\delta_0 > 0$ such that for any $0 < \delta < \delta_0$
the following holds: if $Q$ denotes a nonlinear
excited state with $E -e_1 = \lambda \ep^2$, 
$\L_Q$ denotes the corresponding linearized operator, and 
$\eta_\infty \in W^{2,1} \cap H^2 \cap \E_c(\L_Q)$ with 
$\norm{\eta_\infty}_{W^{2,1} \cap H^2} \leq \delta$, 
then there exists a solution $\psi(x,t)$ of \eqref{NSE} such that
\[ 
  \norm{\psi(x,t) - \psi_{as}(x,t)}_{H^2} \leq C(\ep) 
  \delta^{7/4}(1+t)^{-1},
\]
where
\[ 
  \psi_{as}(x,t) := e^{-i E t} [Q + e^{t\L_Q} \eta_\infty]
\]
and so in particular, for $p > 2$,
\[
  \| \psi(x,t) - e^{-iEt} Q(x) \|_{L^p} \to 0
  \;\; \mbox{ as } \;\; t \to \infty.
\] 
\end{theorem}

\section{Some notation and definitions} \label{ND}

\begin{itemize}

\item[\textup{(1)}] 
Let $L^2 := L^2(\mathbb{R}^3, \mathbb{C})$, $\E := L^2(\mathbb{R}^3, \mathbb{C}^2)$. We equip $\E$ with the standard inner product
\[ \wei{f,g} = \int_{\mathbb{R}^3}( \bar{f}_1 g_1 + \bar{f}_2 g_2) dx , \quad \forall \ f = \begin{bmatrix} f_1 \\ f_2 \end{bmatrix}, \quad g = \begin{bmatrix} g_1 \\ g_2 \end{bmatrix} \in \E.\]
Moreover, for any $u \in L^2$, we shall write $\vt{u} = \begin{bmatrix} \Re (u) \\ \Im(u)\end{bmatrix} \in \E$.

\medskip

\item[\textup{(2)}] 
Recall 
\[
\begin{split}
  & H_0 = -\Delta + V \\
  &\bV := Null(H_0 - e_1) = \text{span}_{\mathbb{C}} 
  \{\phi_1, \phi_1, \phi_2\}, \quad
  Null(H_0 - e_0) = \mathbb{C} \phi_0 \\
  &\bP_0 = \langle \; \phi_0 \; |  \; , \quad
  \bP_1 = \sum_{j=1}^3 \langle \; \phi_j \; | \;, \quad
  \bP_c = 1 - P_0 - P_1
\end{split}
\]

\medskip

\item[\textup{(3)}] 
Denote
\begin{equation} \label{J-sigma.def}
J = \begin{bmatrix} 0 & 1 \\ -1 & 0 \end{bmatrix}, \quad \sigma_1 = \begin{bmatrix} 0 & 1 \\ 1 & 0 \end{bmatrix}, \quad \sigma_2 = \begin{bmatrix} 0 & -i \\ i & 0 \end{bmatrix}, \quad
\sigma_3 = \begin{bmatrix} 1 & 0 \\ 0 & -1 \end{bmatrix}.
\end{equation}

\medskip

\item[\textup{(4)}] 
Let $\bO(k)$ denote the group of orthogonal transformations
on $\R^k$. Identifying $\R^2$ with $\mathbb{C}$, 
we can write $\bO(2)$ 
as the group generated by $\{ e^{i r}, \conj: r \in [0,2\pi)\}$, 
where $\conj$ denotes complex conjugation.
\begin{definition} 
Let $\G:= \bO(3) \oplus \bO(2)$. For $g = (g_1, g_2) \in G$, 
define its action on $L^2$ by $g * f(x) : = g_2* f(g_1*x)$
where $g_1*x$ denotes usual matrix multiplication,
and $g_2 * $ denotes complex multiplication (or conjugation).
\end{definition}

\medskip

\item[\textup{(5)}] \label{Rota} 
For $\al \in [0, 2\pi)$, we denote by $R_{jk}(\al)$ 
the rotation matrix in the $x_j x_k$-plane of $\R^3$ through angle 
$\al$, for $j,k =1,2,3$, and $j < k$. Precisely, 
\begin{equation} \label{basic.rotation}
\begin{split}
R_{12}(\al) & := \begin{bmatrix} \cos(\al) & -\sin(\al) & 0 \\ \sin(\al) & \cos(\al) & 0 \\ 0 & 0 & 1 \end{bmatrix},
\quad R_{13}(\al) := \begin{bmatrix} \cos(\al) &  0 & \sin(\al)\\ 0 & 1 & 0 \\ -\sin(\al) & 0 & \cos(\al)  \end{bmatrix}, \\
R_{23}(\al) & := \begin{bmatrix} 1 &  0 & 0 \\ 0 & \cos(\al) & -\sin(\al) \\ 0 & \sin(\al) & \cos(\al)  \end{bmatrix}.
\end{split}
\end{equation}
Also, let
\begin{equation} \label{Gamma.def}
\begin{split}
& \Gamma(\delta,\al,\sigma) := R_{12}(\delta)R_{13}(\al)R_{23}(\sigma), \quad \Gamma_0(\al, \delta) : = R_{12}(\al)R_{13}(\delta), \\
& \Gamma_1(\al, \delta) := R_{13}(\al)R_{23}(\delta), \quad \al, \delta, \sigma \in [0,2\pi).
\end{split}
\end{equation}
By the Euler-Brauer resolution of a rotation, for any rotation matrix 
$A \in \mathbf{SO}(3)$, there exist unique 
$(\delta, \al, \sigma) \in [0,2\pi)^3$ such that 
$A =\Gamma(\delta,\theta,\sigma)$ (eg. \cite[p. 146]{Rao}). 
Moreover, for any $B \in \bO(3)$, either $B$ or $-B$ lies in
$\mathbf{SO}(3)$. 

\medskip

\item[\textup{(6)}] 
For some $s >3$, let $L^2_s$ be the weighted $L^2$ space defined by 
\begin{equation} \label{Ls}
 L^2_s := \{ f: (1+|x|^2)^{s/2} f(x) \in L^2(\mathbb{R}^3, \mathbb{C})\}.  \end{equation}
Then, let $\E_s = L^2_s \times L^2_s$ and $\B := \B(\E_{s}, \E_{-s})$ be the space of all bounded operators from $\E_s$ to $\E_{-s}$.

\end{itemize}


\section{Existence of Nonlinear Excited States} \label{Exi}
Let $\mu := E -e_1$ and
\begin{equation} \label{nl.ext}
F(\mu,Q) := (H_0 -e_1)Q + \la |Q|^2Q - \mu Q.
\end{equation}
Since $V = V(|x|)$, we see that $F$ is invariant under the action the group $\G$, i.e. 
\[ F(\mu, g*Q) = g *F(\mu, Q), \quad  \forall \ g \in \G. \]
As remarked in the introduction, the equation $F(\mu,Q)=0$ 
has a solution for $Q$ near $\bV$ only if 
$\lambda \mu = \lambda(E-e_1) > 0$. Now, we write 
$Q = \ep(v + h)$ where we will take 
$\ep  = \sqrt{\lambda \mu} > 0$ sufficiently small, 
$v \in \bV$ of order one, and $h \in \bV^\perp$. 
Then the equation $F(\mu, Q) =0$ becomes
\begin{equation} \label{hv.eqn} (H_0-e_1)h + \lambda \ep^2|v+h|^2 - \mu(v+h) =0. \end{equation}
Now, applying the projections $\bP_1$ and 
$\bP_1^\perp := 1 - \bP_1$ to the equation \eqref{hv.eqn}, we get
\begin{flalign} \label{nlex.h}
  & h = \lambda \epsilon^2 [H_0-e_1]^{-1} \bP_1^\perp 
  \{h - |v +h|^2(v +h) \} \\ \label{nlex.ker} & 
  \bP_1 [v - |v+h|^2(v +h)] = 0. 
\end{flalign}
By applying the contraction mapping theorem or implicit function 
theorem, we see that for any fixed $c_1>0$, there 
exists a sufficiently small number $\ep_1>0$, such that for all 
$0 \leq \ep < \ep_1$, and each $v \in \bV$ with $\norm{v} < c_1$, 
there is a unique solution $h=h(\epsilon,v) \in H^2$ 
of~\eqref{nlex.h} satisfying 
\[ 
  h(0,v) =0, \quad h_v(0,v) = 0.
\]
Moreover, since $\bP_1 g = g \bP_1 $ for all $g \in \G$, by the uniqueness of the solution $h$ of \eqref{nlex.h}, we also have
\begin{equation} \label{h.inv}
h(\ep, g*v) = g* h(\ep,v), \ \forall\ g \in \G.
\end{equation}
Let 
\[ N(\epsilon,v) := \Pe [v -  |v+ h(\epsilon,v)|^2(v + h(\epsilon,v))]. \]
Then, we have 
\begin{equation} \label{N.inv}
N(\ep, g*v) = g*N(\ep,v), \quad \text{for all} \quad  v \in \bV \quad  \text{and} \quad g \in \G.
\end{equation}
Moreover,
\begin{equation} \label{N.deriv}
 N_v(0,v)w = \Pe [w - 2|v|^2 w - v^2\bar{w}],  \;\; \forall \ w \in \bV. 
\end{equation}
In order to apply the implicit function theorem to solve the equation $N(\ep, v) =0$, we need to find $v_0 \in \bV$ such that 
\begin{equation} \label{imp.cond}
N(0,v_0) = 0, \quad \text{and} \quad N_v(0,v_0) : \bV \rightarrow  \bV \ \text{is invertible}.
\end{equation}
Set $v_0 = \phi_z := z\cdot \phi = z_1 \phi_1 + z_2 \phi_2 + z_2 \phi_3$ for $z \in \mathbb{C}^3$. Then,
$N(0,v_0) =0$ if and only if $(\phi_j, N(0, \phi_z)) =0$ for all
$j=1,2,3$. Set $I = (\phi_1^2, \phi_2^2)$. 
For each $j=1,2,3$, the equation $(\phi_j, N(0, \phi_z)) =0$ becomes
\begin{flalign*}
z_j & =  (\phi_j, |\phi_z|^2\phi_z) = 2 z_j \sum_{l \not =j} |z_l|^2(\phi_j^2, \phi_l^2) + \bar{z}_j \sum_{l} z_l^2 (\phi_j^2, \phi_l^2) \\
\ & = (\phi_1^2, \phi_2^2) \left \{3 |z_j|^2 z_j + \bar{z}_j \sum_{l\not=j} z_l^2 + 2z_l \sum_{l \not = j} |z_l|^2 \right \} \\
\ & = I \left \{2z_j |z|^2 + \bar{z}_j z^2 \right\}, \quad \text{where} \quad z^2 := \sum_{l} z_l^2.
\end{flalign*}
Equivalently,
\begin{equation} \label{bi.re}
  2z_j |z|^2 + \bar{z}_j z^2 =\frac{1}{I} z_j
  \quad\quad \forall \ j = 1,2,3.
\end{equation}
Write $z_j = a_je^{i\alpha_j}$, for $a_j \geq 0$ and $\al_j \in [0,
  2\pi)$. 
For some $j_0$, we have $z_{j_0} \not= 0$, and by applying
a phase rotation we may assume $z_{j_0} \in \R$.
Then dividing~\eqref{bi.re} by $z_{j_0}$, we get
\begin{equation} \label{first.com}
   2|z|^2 + z^2 = \frac{1}{I}. 
\end{equation}
Moreover for any $j$ with $z_j \not=0$, \eqref{bi.re} implies
\begin{equation} \label{nonzero.com}
 2|z|^2 + e^{-2i\al_j} z^2 = \frac{1}{I}, 
\end{equation}
and so, comparing~\eqref{first.com} and~\eqref{nonzero.com}, we see 
that either $z_j = \pm a_j$, or else $z^2 =0$. 
So \eqref{bi.re} is equivalent to
\begin{equation} \label{re.sol}
  |z|^2 = \frac{1}{2I} \quad \text{and} \quad z^2 = 0; \quad\quad
  \text{or} \quad |z|^2 = \frac{1}{3I} \quad \text{and} \quad 
  e^{i\al}z \in \mathbb{R}^3, \mbox{ some } \alpha \in [0,2\pi).
\end{equation} 
Thus we obtain two representative elements of the solutions of 
$N(0,v)=0$:
\begin{lemma} \label{Orbit} 
Let $v_1 = (1/3I)^{1/2}(1,0, 0) \cdot \phi, \;\; 
v_2 = (1/4I)^{1/2}(1, i, 0) \cdot \phi$, 
and let $\bO_1, \bO_2$ be respectively the orbits of $v_1$ and $v_2$
under the action of $\G$ on $\bV$.  The set $\bO:= \bO_1 \cup \bO_2$
contains all non-zero solutions of $N(0,v) = 0$.
\end{lemma}
\begin{proof} First of all, note that $\bV$ is invariant under the
  action of $\G$. So, the action of $\G$ on $\bV$ is
  well-defined. Moreover, from \eqref{N.inv}, we see that $v$ solves
  the equation $N(0,v) =0$ for all $v \in \bO$. 
Now observe that for any $v = z \cdot \phi \in \bV$, we have
\begin{equation}\label{g.ac.z} 
  g*v = g_2*[z' \cdot \phi], \quad \forall\ g = (g_1, 1) \in \G, 
  \quad z' := g_1 z.
\end{equation}
So, to prove that for every $v = z\cdot \phi$ with $z = (z_1,z_2,z_3)
\in \mathbf{C}^3$ such that $N(0,v) =0$ then $v\in \bO$, we shall need
to find some $g = (g_1, g_2) \in \G$  such that 
\[ 
  g_2*[z' \cdot \phi] = v_1 \ \text{or} \ = v_2, 
 \quad \text{with} \quad z' := g_1 z.
\]
So let $v = z\cdot \phi$ with $z = (z_1,z_2,z_3) \in \mathbf{C}^3$ 
such that $N(0,v) =0$. Then, $z$ satisfies \eqref{re.sol}. 
If  $e^{i\al}z \in \mathbb{R}^3$ for some $\al \in [0,2\pi)$ with 
$|z|^2 = (1/3I)^{1/2}$, then, it's simple to see that there exists 
$g:= (g_1,e^{-i\al}) \in \G$ with $g_1 \in \mathbf{SO}(3)$ such 
that $g*v_1 = z\cdot\phi$. So, $z \in \bO_1$. Now, assume that 
$e^{i\al}z \notin \mathbb{R}^3$ for all $\al \in [0,2\pi)$. Then, 
we have $|z|^2 = 1/(2I)$ and $z^2 =0$. We write 
$z_j = a_j e^{i\al_j}$, for $a_j \geq 0$ and $0 \leq \al_j  < 2\pi$ 
for all $j =1,2,3$. Applying a spatial rotation, and then a phase 
rotation, we may assume that $a_1 >0$ and $\alpha_1 = 0$. 
So, $z = (a_1, a_2e^{i\al_2}, a_3e^{i\al_3})$. Now, if  $a_2a_3 =0$ or $\al_2\al_3 =0$, then it is also simple to show that $v \in \bO$. So, we assume that $a_1a_2a_3 \not=0$ and $\al_2\al_3 \not=0$. 
Let $M_1 := R_{23}(\al)^T$ where $R_{23}$ is defined 
in~\eqref{basic.rotation} and choose $\al$ so that 
$a_2\cos(\al) \cos(\al_2) = a_2\sin(\al)\cos(\al_3)$, or in other 
words $\tan(\al) = a_2\cos(\al_2)/[a_3\cos(\al_3)]$. 
Then, $(M_1, 1) \in \G$ and $(M_1,1)*v = v' := z' \cdot \phi$ where $z' := (r_1,i a_2' , a_3'e^{i\al_3'})$ for some $a_2' \in \mathbb{R}$ and $a_3' \geq 0$. Since $z^2=0$ and $|z|^2 = 1/(2I)$, we have $(z')^2 =0$ and $|z'|=1/(2I)$. So,
\[ 
  a_1 -(a_2')^2 + (a_3')^2 e^{2i\al_3'} =0. 
\]
This implies that $e^{2i\al_3'} \in \mathbb{R}$. Therefore,
$\al_3' \in \{ 0, \; \pi/2, \; \pi, \; 3\pi/2 \}$. 
So either $z' = (r_1, ia, ib)$ or  $z' = (r_1, ia ,b)$ 
for some $a, b \in \mathbb{R}$ such that $|z'| = 1/(2I)$. 
From this and by taking $M_2 = R_{23}(\theta')^T$ or $M_2 = R_{13}(\theta')^T$ with appropriate $\theta'$, we see that $(M_2,1)v' = v'' :=z'' \cdot \phi$ with $z'' = (c, id,0)$ for some $c,d \in \mathbb{R}, c \not=0$ and $|z''|^2 = 1/(2I)$ and $(z'')^2 = 0$. Then, it follows that $d \not=0$ and then $v'' \in \bO_2 \subset \bO$. This completes the proof of the lemma.
\end{proof}
Lemma \ref{Orbit} completely solves the first equation 
of~\eqref{imp.cond}. For the second condition in~\eqref{imp.cond}, 
it is noted that 
\[ 
  N(\ep, v)(iv) = 0, \quad \forall \ v \in \bV: N(\ep, v) =0,  
\]
a consequence of the phase invariance.
So, the second condition of \eqref{imp.cond} never holds. To overcome
this and solve the equation $N(\ep, v) =0$, we shall restrict $N$ to
the invariant subspaces of $\bV$ and solve $N =0$ on these subspaces. 
To this end, we introduce the following lemma:
\begin{lemma} \label{Invariant} For each $j =1,2$, let $\bV_j :=
  \text{span}_{\mathbb{R}} \{ v_j\}$ where $v_1, v_2$ are defined in
  Lemma \ref{Orbit}. 
There are subgroups $\G_j \leq \G$ such that $\bV_j$ is the fixed 
subspace of $\bV$ under the action of $\G_j$ on $\bV$. 
In other words, 
\begin{equation} \label{Fix-subspace} 
  \bV_j = \{v \in \bV : g* v = v \quad \forall \ g \in \G_j\}, 
  \;\; j = 1,2. 
\end{equation}
\end{lemma}
\begin{proof} Recall that $\conj$ denotes complex
conjugation:  $\conj *z = \bar{z}$ for $z \in {\mathbb C}$.
Set
\[
  \iota_0 := \left(\begin{bmatrix} -1 & 0 & 0 \\ 0 & 1 & 0 \\ 0 & 0 & 1
  \end{bmatrix} , -1 \right), \;  
  \iota_1 := \left(\begin{bmatrix} 1 & 0 & 0 \\ 0 & -1 & 0 \\ 0 & 0 & 1
  \end{bmatrix} ,   \conj \right), \;
  \iota_2 := \left( \begin{bmatrix} 1 & 0 & 0 \\ 0 & 1 & 0 \\ 0 & 0 & -1
  \end{bmatrix} ,  1 \right) \; \in \G. 
\] 
Moreover, let
\begin{equation*}
g(\al):= (R_{12}(\al), e^{-i\al}) \in \G \quad \text{and} \quad G_1'  : = \left \{ \begin{bmatrix} 1 & 0 \\ 0 & g\end{bmatrix}, \ g \in \mathbf{O}(2) \right \} \leq \bO(3).
\end{equation*}
Then, let $\G_1$ be the subgroup of $\G$ which is generated by 
the subgroup $\G_1' \oplus \{ 1 , \conj \}$ and $\{\iota_0\}$. 
And let $\G_2$ be the subgroup of $\G$ which is generated by 
$\{\iota_1, \iota_2, g(\al),\ \forall \ \al \in [0, 2\pi)\}$.
We shall show that $\bV_j$ is the fixed subspace of the action of 
$\G_j$ on $\bV$ for $j=1,2$. Note that
\[ 
  \{ v \in \bV : (g,1)* v = v \ \text{and} \ \iota_0*v = v, \ \forall\
  g \in \G_1' \} = \text{span}_{\mathbb{C}} \{\phi_1 \}. 
\]
Moreover, for all $z \in \mathbb{C}$, we see that $\conj * z \phi_1 =
z \phi_1$ if and only if $z \in \mathbb{R}$. So, we obtain
\eqref{Fix-subspace} for $j=1$. 
On the other hand, for $v = z\cdot \phi \in \bV$ such that $v$ is 
fixed by the action of $\G_2$, we have $g(\al)* v =v$ for all $\al \in [0,2\pi)$. So,
\begin{equation*}
\left \{ \begin{array}{ll}
e^{i\al} z_1 & = z_1\cos(\al) + z_2\sin(\al), \\
e^{i\al} z_2 & = -z_1\sin(\al) + z_2\cos(\al), \\
e^{i\al} z_3 & = z_3,  \quad \quad \forall \ \al \in [0,2\pi).
\end{array} \right.
\end{equation*}
Therefore, we get $z_2 = iz_1$ and $z_3 =0$. So $v = z_1 (\phi_1 + i \phi_2)$. Moreover, $\iota_1 *v = v, \iota_2*v =v$ imply that $z_1 \in \mathbb{R}$. So, $v \in \bV_2$ and this completes the proof of the lemma.
\end{proof}
An elementary but important observation: \eqref{N.inv}
implies that $N = N(\ep, \cdot)$ maps $\bV_j$ into $\bV_j$.
Thus we have well-defined maps
\[ 
  N_1 : = N_{| \bV_1} : \bV_1 \to \bV_1, 
 \quad\quad 
  N_2 : = N_{| \bV_2} : \bV_2 \to \bV_2.
\]

Now consider the bifurcation equation $N_1=0$. From the definition of $v_1$ in Lemma \ref{Orbit}, we have $N_1(0, v_1) = 0$. Moreover, from \eqref{N.deriv}
\[ \partial_v N_1(0,v_1)(\phi_1)  = (\phi_1, N_v(0,v_1) \phi_1)\phi_1 =  - 2\phi_1. \]
Therefore, applying the implicit function theorem, we obtain a bifurcation solution of the equation $N_1 = 0$ as
\begin{equation} \label{bir.1}
  v_+(\epsilon) = v_1 + a_+(\epsilon) \phi_1, 
  \quad a_+(\epsilon) = O(\epsilon)
\end{equation}
for $0 < \ep < \ep_2$, some $0 < \ep_2 \leq \ep_1$.
Now, we consider the bifurcation equation $N_2 = 0$. 
Let $\phi^* = \phi_1 + i \phi_2$. As before, we have $N_2(0,v_2) = 0$ and
\[ 
  \partial_v N_2(0,v_2)(\phi^*)  = (\phi^*, N_v(0,v_2) \phi^*)\phi^* =
  [2 - \frac{3}{4I}(|\phi^*|^2,|\phi^*|^2)] \phi^* = -4\phi^*. 
\]
Again, by the implicit function theorem, we have a bifurcation
solution of the equation $N_2 = 0$ given by
\begin{equation} \label{bir.2}
  v_-(\epsilon) = v_2 + a_-(\epsilon) \phi^*, 
  \quad a_-(\epsilon) = O(\epsilon)
\end{equation}
for $ 0 < \ep < \ep_2$, some $0 < \ep_2 \leq \ep_1$.

We denote the two resulting solutions of \eqref{nl.ext} by
\begin{equation} \label{QE.def}
\begin{split}
Q_E & := \epsilon [v_1 + a_+(\epsilon)\phi_1 + h(\epsilon, v_1)] = \rho(\ep) \phi_1 + h(\ep, v_1), \\
\wt{Q}_{E} & = \epsilon [v_2 + a_-(\epsilon)\phi^* + h(\epsilon, v_2)]= \wt{\rho}(\ep)\phi^* + h(\ep, v_2),
\end{split}
\end{equation}
where $\rho(\ep) := \ep/(3I)^{1/2} + \ep a_+(\ep)$ and $\wt{\rho}(\ep) := \ep/(4I)^{1/2} + \ep a_-(\ep)$.
\begin{remark} \label{sym-ex} By \eqref{h.inv}, we see that for each $j=1,2$, we have $h(\epsilon, v_j) = h(\ep, g*v_j) = g*h(\ep, v_j)$ for all $g \in \G_j$. Therefore, we can write $Q_E$ and $\wt{Q}_E$ as
\[ 
  Q_E = x_1 f_1(x_1^2, x_2^2+x_3^2), \quad 
  \wt{Q}_E = e^{i\theta} f_2(x_1^2 +x_2^2, x_3^2)
\]
for some real functions $f_1, f_2$. Here, $\theta \in [0,2\pi)$ 
is the angle between that $x_1$ axis and the vector 
$x' = (x_1,x_2) \in \mathbb{R}^2$.
\end{remark}
Note that $Q_E, \; \wt{Q}_E \in H^2$ by construction.
Their exponential decay is standard; in particular,
the argument in the case of simple eigenvalues
-- see, eg, \cite{GNT} -- applies.
This completes the proof of Theorem~\ref{ex.sol}. 
$\Box$

\begin{remark}
It is straightforward to express the full solution branches 
explicitly in terms of symmetry transformations:
\begin{equation}
\label{parameters}
  Q(x) = e^{i r_1} Q_E( \Gamma_0(r_2,r_3) x ) \quad \mbox{ and } 
  \quad Q(x) = e^{i r_1} \wt{Q}_E( \Gamma_1(r_2,r_3) x)
\end{equation}
for parameters $E \in (e_1, e_1 + \lambda \ep_2^2)$,
$r_1, r_2, r_3 \in [0, 2\pi)$.
\end{remark}


\section{Spectral Analysis of the Linearized Operators Around Excited States} \label{sp.an}

The spectral properties are invariant under symmetry 
transformations of the underlying solution, hence we
may fix one element of each branch of excited states.
So in what follows, let $Q = Q_E$ or $\wt{Q}_E$. 
We write solution $\psi(t,x)$ of \eqref{NSE} as 
$\psi(t,x) = [Q + \zeta]e^{-iEt}$. Then, from \eqref{NSE}, we have
\[ 
  \partial_t \zeta = \L_Q \zeta + \ \text{nonlinear terms},
\]
where
\begin{equation} \label{L.def}
  \L_Q \zeta = -i\{(H_0 -E + 2\lambda |Q|^2) \zeta 
  + \lambda Q^2 \bar{\zeta} \}.
\end{equation}
Decomposing $\zeta$ into the real and imaginary parts, we can extend
$\L$ to $\vt{\L} :\E \rightarrow \E $. That is,
\begin{equation} \label{bL.gamma.def}
  \vt{\L}_Q := J (H_0 - E) + W_Q, \quad \text{where} \quad 
  J := \begin{bmatrix} 0 & 1 \\ -1 & 0 \end{bmatrix}, \quad 
  W_Q : = \begin{bmatrix} W_{1,Q} & W_{2,Q} \\ W_{3,Q} & 
  W_{4,Q} \end{bmatrix}, 
\end{equation}
with
\begin{equation*}
\begin{split}
\ & W_{1,Q}  :=  2 \lambda \Re Q \Im Q, \quad W_{4,Q} = - W_{1,Q}, \\
\ & W_{2, Q} := \lambda(2|Q|^2 + [\Im Q]^2 - [\Re Q]^2 ), \\ 
\ & W_{3, Q} := - \lambda(2|Q|^2 + [\Re Q]^2 - [\Im Q]^2).
\end{split}
\end{equation*}
For $Q = Q_{E}$ we will write $\bL_0 := \vt{\L}_{Q_E}$ and for 
$Q = \wt{Q}_E$, we will write $\bH_0 := \vt{\L}_{\wt{Q}_E}$. 
So
\begin{equation} \label{bH.def}
\begin{split}
& \bL_0 = \begin{bmatrix} 0 & L_{-} \\ -L_+ & 0 \end{bmatrix}, \quad \text{with}\quad  L_\pm := H_0 - E + \lambda( 2 \pm 1) Q^2_E, \\
& \bH_0 =  J (H_0- E) +  \begin{bmatrix}  W_1 & W_2 \\ W_3 & W_4 \end{bmatrix}, \ W_{j} := W_{j,\wt{Q}_E}. 
\end{split}
\end{equation} 

We now state several lemmas which will be used in the next 
sections. These lemmas closely parallel lemmas in~\cite{TY3},
and so we will omit some details.

\begin{lemma} \label{Non-emb-eig} 
Let $\Sigma_c : = \{ia : |a| \geq |E| \}$. 
For all $\tau \in \Sigma_c$, $|\tau| \not= |E|$, the equation 
$\vt{\L} \psi = \tau \psi$ has no non-zero solution $\psi \in \E$. 
Moreover, $\pm i|E|$ are neither eigenvalues nor resonances.
\end{lemma}
\begin{proof} Note that $\vt{\L} = J(-\Delta -E) + \wt{W}$ where 
$\wt{W} = JV +W$. By the assumptions on $V$, and the exponential
decay of $Q$, the potential $\wt{W}$ 
decays quickly. The first part of the lemma then follows 
as in~\cite[Section 2.4]{TY3}, making use of the
resolvent estimate Lemma~\ref{Resol.est}. 
To prove that $\pm i|E|$ is not an eigenvalue or resonance, we may 
use Lemma \ref{Resol.est} and the argument in \cite{JK}, 
or we can follow \cite[Section 2.5]{TY3}.
\end{proof}
\begin{lemma} \label{outside} 
Let $R(z) : = (\vt{\L} -z)^{-1}$, $R_0(z) = [JH -z]^{-1}$ and
$\Sigma_p:= \{0, \pm i(e_0-E)\}$ where $H=H_0-E$.  
There exists an order one constant $C>0$ such that
\begin{equation*}
\begin{split}
& \norm{R(z)}_{(\E,\E)} \leq C[\norm{R_0(z)}_{(\E, \E)} + \norm{R_0(z)}_{(\E, \E)}^2], \ \forall z \notin \Sigma_c : dist(z, \Sigma_p) \geq \ep.
\end{split}
\end{equation*}
\end{lemma}
\begin{proof} We have
\[ 
  R(z) = [1 + R_0(z)W]^{-1} R_0(z) 
  = \sum_{j=0}^\infty(-1)^{j}[R_0(z)W]^jR_0(z). 
\]
Using this and the facts that $R_0(z)$ is uniformly bounded in 
$\B$ with bound $O(1/\ep)$ for $z$ with $dist(z, \Sigma_p) \geq \ep$ (see \cite{JK}), and $W$ is localized of order $\ep^2$, we obtain
\begin{flalign*}
\norm{R(z)}_{(\E, \E)} &  \leq \norm{R_0(z)}_{(\E, \E)} + \sum_{j=1}^\infty C\norm{R_0(z)}_{(\E, \E)} \{C\ep^2 \norm{R_0(z)}_{\B}\}^{j-1}  \norm{R_0(z)}_{(\E, \E)}\\
\ & \leq C\left\{\norm{R_0(z)}_{(\E, \E)} +\norm{R_0(z)}_{(\E, \E)} ^2
\sum_{j=0}^\infty  [C \ep ]^{j} \right \}
\end{flalign*}
For $\ep$ sufficiently small, the series in the right hand side of the
last inequality converges, and the lemma follows.
\end{proof}
\begin{lemma} \label{R0.sp.est} Let $\mathbf{D}_1 :=\{a+ib : |b - (e_0-E)|\leq \ep, 0 < a \leq \ep\}$ and $R_0(z) = (JH -z)^{-1}$. Then for some fixed $s>3$, there is a constant $C>0$ such that for all $z \in \mathbf{D}_1$,
\begin{equation}\label{Ro.est.lo}
\begin{split}
& \norm{\langle x \rangle^{-s} \bP_c R_0(z) \bP_c \langle x \rangle^{-s}}_{(L^2, L^2)} \leq C, \\
& \norm{\langle x \rangle^{-s} \bP_c\frac{d}{dz} R_0(z) \bP_c \langle x \rangle^{-s}}_{(L^2, L^2)} \leq C(\Re z)^{-1/2}.
\end{split}
\end{equation}
Here $\bP_c = \bP_c(JH) = \begin{bmatrix}\bP_c(H) \\ \bP_c(H) \end{bmatrix}$. Moreover, for $z_1, z_2 \in \mathbf{D}_1$, we have
\begin{equation} \label{R0-R0}
 \norm{\langle x \rangle^{-s} \bP_c[ R_0(z_1) - R_0(z_2)] \bP_c \langle x \rangle^{-s}}_{(L^2, L^2)} \leq C[\max \{ \Re z_1, \Re z_2\}]^{-1/2} |z_1 -z_2|. \end{equation}
\end{lemma}
\begin{proof} We write $R_0(z) = R_{01}(z) + R_{02}(z)$, where
\begin{equation} \label{R0.dec} R_{01}(z):= \frac{1}{2}\begin{bmatrix} i & -1 \\ 1 & i\end{bmatrix}(H -iz)^{-1}, \quad R_{02}(z):= \frac{1}{2}\begin{bmatrix} -i & -1 \\ 1 & -i\end{bmatrix}(H + iz)^{-1}.\end{equation}
From this and since $R_{02}$ is regular in $\mathbf{D}_1$, we shall
prove the lemma with $R_0(z)$ replaced by $T_0(z):= (H
-iz)^{-1}\bP_c(H)$. The first estimate in \eqref{Ro.est.lo} is
standard and therefore we skip its proof. The second estimate in
\eqref{Ro.est.lo} follows from \eqref{R0-R0}. So, it's sufficient to
prove \eqref{R0-R0}. The proof now is similar to that of \cite[Lemma
  2.3]{TY3} and we shall only give the main steps. 

For any $z_1, z_2$ in $\mathbf{D}_1$. We write $z_1 = a_1 +ib_1, z_2 = a_2 + ib_2$ and assume that $a_1 \leq a_2$. Let $z_3 = a_2 + ib_1$. We have $z_3 \in \mathbf{D}_1$. For any $u,v \in L^2$ with $\norm{u}_2 = \norm{v}_2 =1$, let $u_1 = \bP_c \wei{x}^{-s}u,\ v_1 = \bP_c\wei{x}^{-s}v$. We have $u_1, v_1 \in L^1 \cap L^2$. Now, using the decay estimate $\norm{e^{-itH}\bP_c \phi}_{L^\infty} \leq C(1+t)^{-3/2}\norm{\phi}_{L^1}$ with $H = H_0 -E$, we get
\begin{flalign*}
& |(u,\wei{x}^{-s}\bP_c[T_0(z_1) -T_0(z_3)]\bP_c \wei{x}^{-s}v)| \\
& \quad = \left |\int_0^{\infty}(u_1, [e^{-it(H-i(a_1+ib_1))} -e^{-it(H-i(a_2+ib_1))}]v_1)dt \right |\\
& \quad = \left |\int_0^{\infty}(u_1, e^{-it(H + b_1)}v_1)(e^{-a_1t} -e^{-a_2t})dt \right |\\
& \quad \leq C\int_0^{\infty}(1+t)^{-3/2}(e^{-a_1t} -e^{-a_2t}) dt \leq Ca_{2}^{-1/2}(a_2 -a_1).
\end{flalign*}
On the other hand, we have
\begin{flalign*}
& |(u,\wei{x}^{-s}\bP_c[T_0(z_3) -T_0(z_2)]\bP_c \wei{x}^{-s}v)| \\
& \quad = \left |\int_0^{\infty}(u_1, [e^{-it(H-i(a_2+ib_1))} -e^{-it(H-i(a_2+ib_2))}]v_1)dt \right |\\
& \quad = \left |\int_0^{\infty}(u_1, e^{-it(H + b_2-ia_2)}v_1)(e^{it(b_2-b_1)} -1)dt \right |\\
& \quad \leq C\int_0^{\infty}(1+t)^{-3/2}e^{-ta_2}|e^{it(b_2-b_1)} -1| dt \leq Ca_{2}^{-1/2}|b_2 -b_1|.
\end{flalign*}
Therefore,
\begin{flalign*}
& |(u,\wei{x}^{-s}\bP_c[T_0(z_1) -T_0(z_2)]\bP_c \wei{x}^{-s}v)| \\
& \quad \leq |(u,\wei{x}^{-s}\bP_c[T_0(z_1) -T_0(z_3)]\bP_c \wei{x}^{-s}v)| + |(u,\wei{x}^{-s}\bP_c[T_0(z_3) -T_0(z_2)]\bP_c \wei{x}^{-s}v)|\\
& \quad \leq Ca_2^{-1/2}(|a_1-a_2| + |b_1-b_2|) \leq C \Re (z_2)^{-1/2} |z_1-z_2|.
\end{flalign*}
This completes the proof of the lemma.
\end{proof}
\subsection{Spectral Analysis of $\bL_0$}
We study the spectrum of $\bL_0$ first. 
From \eqref{J-sigma.def} and \eqref{bH.def}, we have
\begin{equation} \label{adjoint}
  \sigma_1 \bL_0 = \bL^*_0 \sigma_1, \quad \sigma_3 \bL_0 = -\bL_0 \sigma_3, \quad \text{where} 
  \quad \bL^*_0 =  \begin{bmatrix} 0 & -L_+ \\ L_- & 0 \end{bmatrix},
\end{equation}

\begin{lemma}[Spectrum of $L_-$] \label{L-.spectrum} 
For $Q = Q_E$, the following hold: 
\begin{itemize}
\item[\textup{(i)}] 
The discrete spectrum of $L_-$ consists of $\{\wt{e}_0, \wt{e}_{2},
0\}$ where $0$ and $\wt{e}_0 = e_0 -e_1 +O(\epsilon^2)$ 
are simple eigenvalues, and 
$\wt{e}_2 = \wt{e}_3 = -2 \lambda \epsilon^2/3 + O(\epsilon^3)$ 
is a double eigenvalue. For $j = 0,1,2,3$, there exist localized orthonormal functions $\wt\phi_j = \phi_j + O(\epsilon^2)$ such that
\[ L_- \wt \phi_0 = \wt{e_0} \wt\phi_0, \quad L_- \wt \phi_1 =0, \quad L_ - \wt \phi_j = \wt{e}_j \wt\phi_j, \ j = 2,3. \]
Moreover, $\wt \phi_0(x)$ is even with respect to $x_2, x_3$ and $\wt \phi_{j}(x)$ are odd with respect to $x_j$ and even with respect to $x_k$ for all $j, k>1$ and $j \not=k$.
\item[\textup{(ii)}] The continuous spectrum of $L_-$ is $\sigma_c(L_-) = [|E|, \infty)$.
\end{itemize}
\end{lemma}
\begin{proof} We shall just prove (i), as (ii) is standard.
Since $L_-$ is a small perturbation of $H = H_0 -E$, they have the
same number of discrete eigenvalues. 
From Assumption A0, we see that $H$ has a simple eigenvalue $e_0 -E$
with eigenfunction $\phi_0$ and a triple eigenvalue $e_1 -E$ with 
eigenfunctions $\phi_1, \phi_2, \phi_2$. We will compute the 
discrete eigenvalues of $L_-$ which are perturbations of 
the discrete eigenvalues of $H$.

Since $e_0-E$ is simple, by standard perturbation theory, we can find
an eigenfunction $\wt \phi_{0} = \phi_0 + O(\epsilon^2)$ such that
$L_{-} \wt \phi_{0} =  \wt e_{0} \wt \phi_{0}$ and $\wt e_{0} = e_0 -E
+ O(\ep^2) = e_0 - e_1 + O(\epsilon^2)$. 
Moreover, direct computation gives $L_-Q =0$. 
Therefore, with $\wt{\phi}_1 = \frac{1}{\norm{Q}_{L^2}}Q$, 
we have $L_- \wt{\phi}_1 =0$.

Now, we shall show that there exists a double eigenvalue 
$\wt{e}_2 = \wt{e}_3 = -2\ep^2/3 + O(\ep^3)$ of $L_-$. 
In other words, we need to solve $L_- \wt \phi = e \wt \phi$ for 
$e$ close to $e_1 - E = O(\ep^2)$.
Again, we write $\wt \phi = z \cdot \phi + g$, where $g$ is in $\bV^\perp$ and $z \in \mathbb{R}^3$. Then
\begin{equation} \label{ei.Lm}
 (H_0 -e_1) g = (e+ \la(\epsilon^2 - Q^2))(z \cdot \phi + g).    \end{equation}
This equation is solvable if and only if 
\[ 
  (\phi_j, (e+ \la \epsilon^2)z \cdot \phi  - \la Q^2 z \cdot \phi ) 
  = \la(\phi_j, Q^2 g) \quad \forall \ j = 1, 2,3. 
\]
Since $Q^2 = \rho^2 \phi_1^2 + 2\rho \phi_1 h + h^2$ where $h$ 
is a function of $x_1, |x'|$ which is odd in $x_1$, we get 
\begin{equation} \label{e.cond}
  \left \{ e + \la \epsilon^2  - \la[\rho^2(\phi_1^2, \phi_j^2) 
  + 2\rho(\phi_j^2 \phi_1, h) + (\phi_j^2, h^2)]\right \} 
  z_j  = \lambda(\phi_j, Q^2g). 
\end{equation}
Moreover, since $Q_E(x)$ depends on $x_1, |x'|$ and it is odd in $x_1$ for $x=(x_1, x')$, we see that when we restrict $g$ to the class of function which is odd in $x_2$ and even in $x_1, x_3$ we have $(\phi_j, Q^2g) =0$ for $j=1,3$. Solving the equation \eqref{e.cond} in this class of functions $g$, we obtain an eigenvalue 
\[ 
  e =\wt e_{2} := \la \left[ I \rho^2 + 2\rho(\phi_2^2 \phi_1, h) + 
 (\phi_2^2, h^2) - \epsilon^2 \right] + O(\epsilon^4) 
  = -\frac{2}{3} \la \epsilon^2 + O(\epsilon^3)
\]
of $L_-$ with a normalized eigenfunction $\wt \phi_{2} = \phi_2 + g_{2}$, where $g_{2} = O(\epsilon^2)$ which is odd in $x_2$ and even in $x_1, x_3$. 

Now, let $\wt{\phi}_3 = (R_{23}(\pi/2),1) * \wt{\phi}_2$. Since $(R_{23}(\pi/2),1)*Q_E = Q_E$, we see that $L_-\wt{\phi}_3 = \wt{e}_2 \wt{\phi}_3$. Moreover, $\wt{\phi}_3$ is odd in $x_3$ and even in $x_1, x_2$. Therefore, $(\wt{\phi}_2, \wt{\phi}_3) =0$.  So, $\wt{e}_2$ has multiplicity at least two. Then, by counting the number of discrete eigenvalues of $L_-$ and $H$, we see that $\wt{e}_2$ has multiplicity two and we have found all of the discrete eigenvalues of $L_-$. This completes the proof of the lemma.
\end{proof}
\begin{lemma}[Spectrum of $L_+$] \label{L+.spectrum} 
For $Q = Q_E$, the following statements hold 
\begin{itemize}
\item[\textup{(i)}] The discrete spectrum of $L_+$ consists of 
$\{\hat{e}_0, 0, \hat{e}_1\}$ where 
$\hat{e}_0 = e_0 -e_1 +O(\epsilon^2)$ and 
$\hat{e}_1 = 2 \la \epsilon^2 + O(\epsilon^3)$ 
are simple eigenvalues, and $0$ is a double eigenvalue. 
For $j = 0,1,2,3$, there exist localized orthonormal functions 
$\varphi_j = \phi_j + O(\epsilon^2)$ such that
\[ 
  L_+ \varphi_k = \hat{e}_k \varphi_k, \quad \text{for} 
  \ k = 0,1, \text{and} \quad L_+ \varphi_j = 0,  
  \ \text{for} \ j = 2,3. 
\]
Moreover, $\varphi_0(x)$ is even with respect to $x_2, x_3$ and $\varphi_{j}(x)$ are odd with respect to $x_j$ and even with respect to $x_k$ for all $j, k>1$ and $j \not=k$.
\item[\textup{(ii)}] The continuous spectrum of $L_+$ is $\sigma_c(L_+) = [|E|, \infty)$.
\end{itemize}
\end{lemma}
\begin{proof} Again, we only need to prove (i). 
The existence of $\varphi_0$ and $\hat{e}_0$ follows from standard
perturbation theory. Now consider the eigenvalues which are a
perturbation of $e_1 -E$.  Recall that 
$Q_E(\Gamma_0(r_2,r_3) x)$ are all solutions, where 
$r_2,r_3 \in [0, \pi)^2$. 
From \eqref{basic.rotation} and \eqref{Gamma.def}, we have
\begin{equation} \label{gamma0.def}
\Gamma_0 = \Gamma_0(r_2, r_3) =
\begin{bmatrix} \cos(r_2) \cos (r_3) & - \sin(r_2)  & \cos(r_2) \sin(r_3)  \\
\sin(r_2)\cos(r_3)  & \cos (r_2) & \sin(r_2)\sin(r_3) \\
\sin(r_3)  & 0 &  \cos(r_3) 
\end{bmatrix}.
\end{equation}
In particular, 
\begin{equation} \label{mo.deri}
\frac{\partial \Gamma_0(0,0)}{\partial r_2} =
\begin{bmatrix} 0 & -1 & 0 \\
1 & 0 & 0 \\
0 & 0 & 0 \\
\end{bmatrix}, \quad
\frac{\partial \Gamma_0(0,0)}{\partial r_3} =
\begin{bmatrix} 0 & 0 & 1 \\
0 & 0 & 0 \\
-1 & 0 & 0 \\
\end{bmatrix}.
\end{equation}
Now, let
\begin{equation} \label{Z.def}
  Z_1 := \frac{\partial Q_E}{\partial E}, \quad 
  Z_2 := \frac{\partial Q_E(\Gamma_0(r_2,r_3) x)}{\partial r_2}
  \left |_{r_2=r_3=0} \right., \quad 
  Z_3 := \frac{\partial Q_E(\Gamma_0(r_2,r_3) x)}{\partial r_2} 
  \left |_{r_2=r_3=0} \right.  . 
\end{equation}
From \eqref{mo.deri}, we see that
\begin{flalign*}
 Z_1(x) & = \partial_{E} [\rho(\epsilon)] \phi_1 + \partial_E h, \\ 
 Z_2 (x) & = (x_2\partial_1 - x_1\partial_2 )Q_E(x) = \rho(\epsilon) \phi_2 + (x_2\partial_1 - x_1\partial_2 ) h ,\\
 Z_3(x) & = (x_3 \partial_1 - x_1 \partial_3) Q_E(x) = \rho(\epsilon) \phi_3 +(x_3 \partial_1 - x_1 \partial_3) h.
\end{flalign*}
Since $Q_E(x)$ is even with respect to $x_2$ and $x_3$, so is $Z_1$. Moreover, $Z_2$ is odd with respect to $x_2$ and even with respect to $x_3$ and $Z_3$ is odd with respect to $x_3$ and even with respect to $x_2$. Therefore, we have
\begin{equation} \label{gr1.orth}
(Z_1, Z_2) = (Z_1, Z_3) = (Z_2, Z_3) = 0.
\end{equation}
Now, recall the equation~\eqref{nl.ext} for the excited states:
\[ 
  (H_0 -E) Q + \lambda |Q|^2 Q = 0. 
\]
Taking $Q = Q_E(\Gamma_0(r_2,r_3) x)$ and differentiating this 
equation with respect to $E$, $r_2$ and $r_3$, we find
\[ 
  L_+ Z_1 = Q_E, \quad L_+ Z_2 = L_+ Z_3 = 0. 
\]
So 
\[ 
  \varphi_2 = \frac{1}{||Z_2||_{L^2}}Z_2, \quad 
  \varphi_3 = \frac{1}{||Z_3||_{L^2}}Z_3 
\]
are (orthonormal) zero-eigenfunctions.
The computation of $\varphi_1$ and $\hat{e}_1$ is exactly the same 
as in Lemma~\ref{L-.spectrum}. In particular, as in~\eqref{e.cond}, 
we get
\[ 
  \hat{e}_1 =  3 \la[3I \rho^2 + 2\rho(\phi_1^3, h) + (\phi_1^2, h^2)] 
  - \la \epsilon^2 + O(\epsilon^4) = 2 \la \epsilon^2 + O(\epsilon^3).
\]
This completes the proof of the lemma.
\end{proof}
Now, with $Z_j$ as in \eqref{Z.def}, we have 
\begin{equation} \label{gr1.eign}
  \bL_0 \begin{bmatrix} 0 \\ Q_E
  \end{bmatrix} = 0, \quad 
  \bL_0 \begin{bmatrix} Z_1 \\ 0
  \end{bmatrix} =  \begin{bmatrix} 0 \\ Q_E 
  \end{bmatrix}, \quad 
  \bL_0 \begin{bmatrix} Z_2 \\ 0
  \end{bmatrix} = \bL_0 \begin{bmatrix} Z_3 \\ 0
  \end{bmatrix} = 0.
\end{equation}
Moreover, from Lemma \ref{L-.spectrum}, we have $\text{kernel}(L_-) = \text{span} \{Q_E\}$. Since $(Q_E, Z_2) = (Q_E, Z_3) =0$, there are two functions $Y_2$ and $Y_3$ such that 
\begin{equation} \label{Y23.def}
L_- Y_2 = Z_2 \quad L_- Y_3 = Z_3 \quad \text{and} \quad (Q_E, Y_2) = (Q_E, Y_3) = (Y_2, Y_3) = 0. 
\end{equation}
\begin{lemma} \label{inner}  For all $ j =1,2,3$ let $\alpha_j := -(Z_j, Y_j)^{-1}$ where $Y_1 = Q_E$. Then, $\alpha_j$ is finite for all $j =1,2,3$.
\end{lemma}
\begin{proof} We need to show that $(Z_j,Y_j) \not=0$ for all $j
  =1,2,3$. For $j =1$, we have $(Y_1, Z_1) = (Q_E, \partial_E Q_E) =
  (6I)^{-1} + O(\epsilon) \not =0$ for $\epsilon$ sufficiently small. 
Now, we shall show that $(Z_j, Y_j) \not=0$ for $\epsilon$
  sufficiently small. Since the proof for the case $j =2$ is identical
  to that 
of the case $j=3$, we only need to prove $(Z_2, Y_2) \not=0$. 
Recall that $Z_2 = a(\epsilon) \phi_2 + \wt{h}$ for some localized order $\epsilon^3$ function $\wt{h}$ which is odd in $x_2$ and even in $x_3$. We write
\[ 
  Y_2 = \sum_{j=0}^3 b_j \wt{\phi}_j + g, \quad g \perp 
  \wt{\phi}_j, \ \forall \ j = 0,1,2,3. 
\]
Here $\wt{\phi}_j$ are the eigenfunctions of $L_-$ defined in 
Lemma~\ref{L-.spectrum}. By choosing $Y_2$ not to have a 
component in the kernel of $L_-$, we can assume $b_1 =0$. 
Moreover, since $L_- Y_2 = Z_2$ we have
\[ 
  L_-g = - \sum_{j=0}^3 \wt{e_j} \wt \phi_j + 
  \rho(\epsilon) \phi_2 + \wt{h}. 
\]
Taking the inner product of this equation with $\wt{\phi}_j$ for $j =0,1,2,3$ and using the symmetry properties of these functions, we get
\[ \wt{e_2}b_2 = \rho(\epsilon)(\wt{\phi}_2, \phi_2) + (\wt{\phi}_2, \wt{h}) = \rho(\epsilon) + O(\epsilon^3), \quad b_j = 0, \ \forall \ j \not= 2. \]
Since $\wt{e}_2 = -2 \la \epsilon^2/3 + O(\epsilon^3)$, $\wt{\phi}_2 = \phi_2 + O(\epsilon^2)$ we see that 
\[ 
  b_2 = -\frac{3 \la \epsilon^{-1}}{2(3I)^{1/2}}  + O(1). 
\]
To sum up, we have $Y_2 = (\epsilon^{-1}) \wt{\phi}_2 + g$ for $g \in \Pc(L_-)$ and satisfies
\[ L_ - g = O(\epsilon^3) \phi_2 + O(\epsilon^3) = O(\epsilon^3). \]
This implies that $g = O(\epsilon)$. Therefore $(Y_2, Z_2) = O(1) + O(\epsilon) \not=0$ for $\epsilon$ sufficiently small. This completes the proof of the Lemma. 
\end{proof}
Now, we have the following theorem on the spectrum of $\bL_0$:
\begin{theorem}[Invariant Subspaces of $\bL_0$ - Resonant Case] 
\label{Spectrum.R} 
Assume that $e_0 < 2e_1$. The space 
$\E := L^2(\mathbb{R}^3, \mathbb{C}^2)$ 
can be decomposed as the direct sum of $\bL_0$-invariant subspaces 
\begin{equation} \label{E.R.dec}
\E = \oplus_{j=1}^3 \E_{0j} \oplus \E_+ \oplus \E_- \oplus \E_c.
\end{equation}
If $f$ and $g$ belong to different subspaces, then $\wei{J f, g} =0$. These subspaces and their corresponding projections satisfy the following:
\begin{itemize}
\item[\textup{(i)}] For each $j =1,2,3$, let 
\[
  \Phi_{0j} := \frac{\partial}{\partial r_j}
  \vt{e^{-i r_1} Q_E}(\Gamma_0(r_2,r_3) x) \big|_{r_1=r_2=r_3=0}. 
\]
Also, let $\wt{\Phi}_{01} = \frac{\partial Q_E}{\partial E}$ 
and $\wt{\Phi}_{0k} := \begin{bmatrix}0 \\ Y_k \end{bmatrix}$ for $k=2,3$. Then, we have $\bL_0 \wt{\Phi}_{0j} = \Phi_{0j}, \ \bL_0 \Phi_{0j} =0$ and
\[ \alpha_j^{-1} = \wei{J \Phi_{0j}, \wt{\Phi}_{0j}}, \quad \wei{J \Phi_{0j}, \Phi_{0j}} = \wei{J \wt{\Phi}_{0j}, \wt{\Phi}_{0j}} =0. \] 
Moreover, the subspace $\E_{0j}$ is spanned by $\{\Phi_{0j}, \wt{\Phi}_{0j}\}$ with the projection $\bP_{0j}$ from $\E$ onto $\E_{0j}$  defined by 
\begin{equation*} 
\bP_{0j}f = -\alpha_j\wei{J \wt{\Phi}_{0j}, f} \Phi_{0j}  + \alpha_j\wei{J\Phi_{0j}, f} \wt{\Phi}_{0j}, \ \forall \ f \in \E, \ \ j = 1,2,3. 
\end{equation*}
\item[\textup{(ii)}] 
There are four single eigenvalues 
$\omega_1, \omega_2 := -\bar{\omega}_1, \omega_3 := \bar{\omega}_1,
\omega_4 := -\omega_1$ of $\bL_0$ with corresponding eigenvectors 
$\Phi_1, \Phi_2 : = \sigma_3 \bar{\Phi}_1, \Phi_3: = \bar{\Phi}_1,
\Phi_4:= \sigma_3 \Phi_1$ where $\Phi_1 := \begin{bmatrix} \phi_0 \\
  -i \phi_0 \end{bmatrix} + h_1$, $\omega_1 := i\kappa +  \gamma$ with
$\kappa := e_1 - e_0 + O(\epsilon^2) >0$ and $\gamma := \gamma_0
\epsilon^4 + O(\epsilon^5)$ for some positive order one constant
$\gamma_0$. The function $h_1$ satisfies $\norm{h_1}_{L^p} \leq
C[\ep^2 + \ep^{6-\frac{12}{p}}]$ for all $1 \leq p \leq \infty$ and
$\norm{h_1}_{L^2\loc} \leq C\ep^2$ and $(J\Phi_1,\Phi_1)=0$. 
The subspaces $\E_+, \; \E_-$ are defined as
\begin{equation*}
  \E_+ := \text{span}_{\mathbb{C}} \left \{ \Phi_1, \Phi_2\right \} 
  , \quad \E_- := \text{span}_{\mathbb{C}} 
  \left \{ \Phi_3, \Phi_4 \right \}.
\end{equation*}
The projections of $\E$ onto $\E_+$ and $\E_-$ are defined by
\begin{equation*} 
\begin{split}
\bP_+ f &= c_1\wei{J\Phi_1,f} \Phi_2 -\bar{c}_1\wei{J\Phi_2,f}\Phi_1, \\
\bP_- f &= \bar{c}_1\wei{J\Phi_3,f} \Phi_4 - c_1\wei{J\Phi_4,f}\Phi_3, \quad \text{with} \quad c_1 := \wei{J\Phi_1, \Phi_2}^{-1}.
\end{split}
\end{equation*}
\item[\textup{(iii)}] $\E_c = \{g: \wei{J f, g} =0, \quad \forall f \in \oplus_{j=0}^3\E_{0j} \oplus \E_+ \oplus \E_- \}$. Its corresponding projection is $\Pc(\bL_0) = Id - \sum_{j=1}^3\bP_{0j}(\bL_0) - \bP_1^+(\bL_0) - \bP_1^-(\bL_0)$.
\end{itemize}
\end{theorem}
\medskip

\setlength{\unitlength}{1mm}\noindent
\begin{center}
\begin{picture}(120,80)

\put (0,40){\vector(1,0){120}}
\put (120,42){\makebox(0,0)[c]{\scriptsize x}}

\put (60,0){\vector(0,1){80}}
\put (60,82){\makebox(0,0)[c]{\scriptsize y}}

\multiput(60, 76)(0.5, 0){30}{\circle*{0.1}}\put (75,76){\vector(1,0){1}}
\put(82,76){\makebox(0,0)[c]{\scriptsize$\sigma_c(\bL_0)$}}

{\color{red}\put (60.15,60){\line(0,1){18}}\put (59.8,60){\line(0,1){18}}}

\put(64,60){\makebox(0,0)[c]{\scriptsize$|E|$}}

{\color{red}\put (60.15,20){\line(0,-1){20}}\put (59.8,20){\line(0,-1){20}}}
\put(64,20){\makebox(0,0)[c]{\scriptsize$-|E|$}}

\put(60,40){\circle*{1.5}}

\multiput(60, 40)(0.8, -0.4){30}{\circle*{0.1}}\put (84,28){\vector(2,-1){1}}
\put(100,25){\makebox(0,0)[c]{\scriptsize$\E_0 =\{\Phi_{0j}, \wt{\Phi}_{0j}, \ j = 1,2, 3\}$}}

\put(40,72){\makebox(0,0)[c]{\scriptsize$\E_{+}= \{\Phi_1, \Phi_2\}$}}

\put(62,66){\circle*{1.5}}
\put(80,66){\makebox(0,0)[c]{\scriptsize$(\omega_1, \Phi_1),\ \Re(\omega_1) = O(\ep^4)$}}

\put(58,66){\circle*{1.5}}
\put(41,66){\makebox(0,0)[c]{\scriptsize$(\omega_2 =-\bar{\omega}_1, \Phi_2 =\sigma_1\bar{\Phi}_1)$}}

\put(62,14){\circle*{1.5}}
\put(77,14){\makebox(0,0)[c]{\scriptsize$(\omega_3 =\bar{\omega}_1, \Phi_3= \bar{\Phi}_1)$}}

\put(58,14){\circle*{1.5}}
\put(43,14){\makebox(0,0)[c]{\scriptsize$(\omega_4=\bar{\omega}_2, \Phi_4 = \bar{\Phi}_2)$}}

\put(40,8){\makebox(0,0)[c]{\scriptsize$\E_{-}= \{\Phi_3, \Phi_4\}$}}

\multiput(60, 4)(0.5, 0){30}{\circle*{0.1}}\put (75,4){\vector(1,0){1}}
\put(82,4){\makebox(0,0)[c]{\scriptsize$\sigma_c(\bL_0)$}}

\end{picture}

Figure 1: Spectrum of $\bL_0$ in the resonant case.
\end{center}
\begin{proof} 
From \eqref{gr1.eign}, we have $\bL_0 \Phi_{0j} =0$ and 
$\bL_0 \wt{\Phi}_{0j} = \Phi_{0j}$ for all $k=1,2,3$. So, 
from \eqref{adjoint}, \eqref{gr1.eign}, \eqref{Y23.def} and 
Lemma \ref{inner}, the statement (i) follows. Moreover, we have
\[ 
  {\bL_0}_{ \left |_{\E_0} \right.} = 
  \begin{bmatrix} 0 & -1 & 0 & 0 & 0 & 0 \\
  0 & 0 & 0 & 0 & 0 & 0 \\
  0 & 0 & 0 & 1 & 0 & 0 \\
  0 & 0 & 0 & 0 & 0 & 0 \\
  0 & 0 & 0 & 0 & 0 & 1 \\
  0 & 0 & 0 & 0 & 0 & 0 \\
\end{bmatrix}. 
\]
The proofs of (ii) and (iii) are similar and simpler than the 
proofs of (iii) and (iv) of Theorem~\ref{Spectrum.C.R} below, 
so they are omitted.
\end{proof}

Now, if we denote 
\[
  \bL_r :=  \mbox{ the linearized operator around excited state }
  \;\; e^{i r_1} Q_E(\Gamma_0(r_2,r_3) x),
\]
for $r = (r_1, r_2, r_3) \in \R^3$, then
\[ 
  \bL_r \left( r*\begin{bmatrix} u \\   v \end{bmatrix}\right) =
  r* \left( \bL_0 \begin{bmatrix} u \\   v \end{bmatrix}\right), \
  \text{where}\ r* \begin{bmatrix} f(x) \\   g(x)\end{bmatrix} :=
  e^{Jr_1}\begin{bmatrix} f(\Gamma_0(r_2,r_3)x) \\
  g(\Gamma_0(r_2,r_3)x) \end{bmatrix}. 
\]
This observation leads to the following corollary.
\begin{corollary} \label{Cor.R} For $r=(r_1,r_2,r_3) \in \mathbb{R}^3$ and for $j=1,2,3$, let $\Phi_{0j}^r : = r* \Phi_{0j}$ and $\wt{\Phi}_{0j}^r = r* \wt{\Phi}_{0j}$. Similarly, for $j = 1,2,3,4$, let $\Phi_j^r := r*\Phi_j$. Also, let
\begin{flalign*}
\E_{j0}^r & := \text{span}\{\Phi_{0j}^r, \wt{\Phi}_{0j}^r\} \ j =  1,2,3, \\
\E_+^r & := \text{span}\{\Phi_1^r, \Phi_2^r \}, \ \E_-^r := \text{span}\{\Phi_3^r, \Phi_4^r \} \ \forall \ j = 1,2, 3,\\
\E_c^r & := \{f \in \E\ : \wei{Jf, g} = 0, \ \forall \ g \in \oplus_{j=1}^3 \E_{0j}^r \oplus \E_-^r \oplus \E_+^r \}.
\end{flalign*}
Then $\E_{0j}^r, \E_k^r$ are invariant under $\bL_r$, 
$\E = \oplus_{j=1}^3 \E_{0j}^r \oplus \E_-^r \oplus \E_+^r \oplus
\E_c^r$, and the projections from $\E$ into these subspaces are
defined exactly as the corresponding projections in Theorem
\ref{Spectrum.R}. 
Moreover, 
\[ 
  \bL_r \wt{\Phi}_{0j}^r = \Phi_{0j}^r, \ 
  \bL_r \Phi_{0j}^r =0, \ 
  \bL_r \Phi_k^r = \omega_k \Phi_k^r,  \  
  \forall \ j = 1,2,3, \ k = 1,2, 3,4. 
\]
\end{corollary}
Analogous to Theorem~\ref{Spectrum.R}, we have the following 
for the non-resonant case:
\begin{theorem}[Invariant Subspaces of $\bL_0$ -- Non-Resonant Case] \label{Spectrum.NR} Assume that $e_0 > 2e_1$. Then, the space $\E := L^2(\mathbb{R}^3, \mathbb{C}^2)$ can be decomposed as the direct sum of $\bL_0$-invariant subspaces 
\begin{equation} \label{E.NR.dec}
\E = \oplus_{j=1}^3 \E_{0j} \oplus \E_1 \oplus \E_2 \oplus \E_c.
\end{equation}
If $f$ and $g$ belong to different subspaces, then $\wei{J f, g} =0$. These subspaces and their corresponding projections satisfy the followings
\begin{itemize}
\item[\textup{(i)}] 
For each $j =1,2,3$, let 
\[
  \Phi_{0j} :=\frac{\partial}{\partial r_j} 
  \vt{e^{-i r_1} Q_E}(\Gamma_0(r_2,r_3) x)|_{r_1=r_2=r_3=0}. 
\]
Also, let 
$\wt{\Phi}_{01} = \frac{\partial \vt{Q_E}}{\partial E}$ and 
$\wt{\Phi}_{0k} := \begin{bmatrix}0 \\ Y_k \end{bmatrix}$ for $k=2,3$. Then, we have $\bL_0 \wt{\Phi}_{0j} = \Phi_{0j}, \ \bL_0 \Phi_{0j} =0$ and
\[ \alpha_j^{-1} = \wei{J \Phi_{0j}, \wt{\Phi}_{0j}}, \quad \wei{J \Phi_{0j}, \Phi_{0j}} = \wei{J \wt{\Phi}_{0j}, \wt{\Phi}_{0j}} =0. \] 
Moreover, the subspace $\E_{0j}$ is spanned by $\{\Phi_{0j}, \wt{\Phi}_{0j}\}$ with
the projection $\bP_{0j}$ from $\E$ onto $\E_{0j}$  defined by 
\begin{equation*} 
\bP_{0j}f = -\alpha_j\wei{J \wt{\Phi}_{0j}, f} \Phi_{0j}  + \alpha_j\wei{J\Phi_{0j}, f} \wt{\Phi}_{0j}, \ \forall \ f \in \E, \ \ j = 1,2,3. 
\end{equation*}
\item[\textup{(ii)}] 
There are two single purely imaginary eigenvalues 
$\omega_1, \omega_2 := -\omega_1$ of $\bL_0$ with eigenvectors 
$\Phi_1$ and $\Phi_2 := \bar{\Phi}_1$ where 
$\Phi_1 = \begin{bmatrix} u \\ -i v \end{bmatrix}$, $\omega_1 =
i\kappa$ for some constant 
$\kappa = e_1 -e_0  + O(\epsilon^2) >0$. 
The functions $u$ and $v$ are real with $(J\Phi_1, \Phi_1) = 2i(u,v)
=i$ and satisfy $L_+ u = -\kappa v, L_-v = -\kappa u$. 
For each $j =1,2$, the subspace $\E_j$ is spanned by $\Phi_j$, 
and the projection of $\E$ onto $\E_j$ is defined by
\begin{equation*} 
\begin{split}
  \bP_j(\bL_0) f &=  (-1)^ji\wei{J \Phi_j, f}\Phi_j , \ \ 
  \forall \ f \in \E.
\end{split}
\end{equation*}
\item[\textup{(iii)}] 
$\E_c = \{g: \wei{J f, g} =0, \quad \forall f \in \oplus_{j=1}^{3} \E_{0j}  \oplus \E_1  \oplus \E_2\}$. Its corresponding projection is $\Pc(\bL_0) = Id - \sum_{j=1}^3 \bP_{0j}(\bL_0) - \sum_{j=1}^2 \bP_j(\bL_0)$.
\end{itemize}
\end{theorem}
\medskip

\setlength{\unitlength}{1mm}\noindent
\begin{center}
\begin{picture}(120,80)

\put (0,40){\vector(1,0){120}}
\put (120,42){\makebox(0,0)[c]{\scriptsize x}}

\put (60,0){\vector(0,1){80}}
\put (60,82){\makebox(0,0)[c]{\scriptsize y}}

{\color{red}\put (60.15,60){\line(0,1){18}}\put (59.8,60){\line(0,1){18}}}

\put(64,60){\makebox(0,0)[c]{\scriptsize$|E|$}}

{\color{red}\put (60.15,20){\line(0,-1){20}}\put (59.8,20){\line(0,-1){20}}}
\put(64,20){\makebox(0,0)[c]{\scriptsize$-|E|$}}

\multiput(60, 76)(0.5, 0){30}{\circle*{0.1}}\put (75,76){\vector(1,0){1}}
\put(83,76){\makebox(0,0)[c]{\scriptsize$\sigma_c(\bL_0)$}}

\put(60,40){\circle*{1.5}}

\multiput(60, 40)(0.8, -0.4){30}{\circle*{0.1}}\put (84,28){\vector(2,-1){1}}
\put(100,25){\makebox(0,0)[c]{\scriptsize$\E_0 =\{\Phi_{0j}, \wt{\Phi}_{0j}, \ j = 1,\cdots 3\}$}}

\put(60,56){\circle*{1.5}}
\put(81,56){\makebox(0,0)[c]{\scriptsize$(\omega_1 =i(e_1-e_0)+O(\ep^2), \Phi_1)$}}

\put(60,24){\circle*{1.5}}
\put(45,24){\makebox(0,0)[c]{\scriptsize$(\omega_2 =\bar{\omega}_3, \Phi_2 = \bar{\Phi}_3)$}}

\multiput(60, 4)(0.5, 0){30}{\circle*{0.1}}\put (75,4){\vector(1,0){1}}
\put(82,4){\makebox(0,0)[c]{\scriptsize$\sigma_c(\bL_0)$}}

\end{picture}

Figure 2: Spectrum of $\bL_0$ in the non-resonant case.
\end{center}
\begin{corollary} \label{Cor.NR} For $r=(r_1,r_2,r_3) \in \mathbb{R}^3$ and for $j=1,2,3$, let $\Phi_{0j}^r : = r* \Phi_{0j}$ and $\wt{\Phi}_{0j}^r = r* \wt{\Phi}_{0j}$. Similarly, for $j = 1,2$, let $\Phi_j^r := r*\Phi_j$. Also, let
\begin{flalign*}
\E_{j0}^r & := \text{span}\{\Phi_{0j}^r, \wt{\Phi}_{0j}^r \} \ j =  1,2,3, \\
\E_k^r & := \text{span}\{\Phi_k^r \}, \ \forall  \ k =1,2, \\
E_c^r & := \{f \in \E\ : \wei{Jf, g} = 0, \ \forall \ g \in \oplus_{j=1}^3 \E_{0j}^r \oplus \E_1^r \oplus \E_2^r \}
\end{flalign*}
Then $\E_{0j}^r, \E_k^r$ are invariant under 
$\bL_r$, $\E = \oplus_{j=1}^3 \E_{0j}^r \oplus \E_1^r \oplus \E_2^r
\oplus \E_c^r$, 
and the projections from $\E$ into these subspaces are defined exactly as those corresponding projections in Theorem \ref{Spectrum.NR}. Moreover,
\[ 
  \bL_r \wt{\Phi}_{0j}^r = \Phi_{0j}^r, \ 
  \bL_r \Phi_{0j}^r =0, \ 
  \bL_r \Phi_k^r = \omega_k \Phi_k^r,  \  
  \forall \ j = 1,2,3, \ k = 1,2. 
\]
\end{corollary}

\subsection{Spectral Analysis of $\bH_0$}

Recall \eqref{QE.def} that 
$\wt{Q}_E = \epsilon[ v_2 + h_2(\epsilon) + h(\epsilon, v_2)]$ 
where $v_2 = 1/(2\sqrt{I}) [\phi_1 + i \phi_2]$, $h_2(\epsilon) =
O(\epsilon) v_2$ and $h(\epsilon, v_2) \in \bV^\perp$. 
In other word, we can write 
$\wt{Q}_E = \wt{\rho}(\epsilon)[\phi_1 + i \phi_2] + \wt{h}$ where
$\wt{\rho}(\epsilon) = \epsilon/(2\sqrt{I}) +O(\epsilon^2)$ and
$\wt{h} \in \bV^\perp$ and of order $O(\epsilon^3)$. 
From \eqref{bL.gamma.def}, the linearized operator around
$\wt{Q}_E$ is
\begin{equation} \label{bHr.def}
  \bH_0 =  J (H_0-E) +  W_{\wt{Q}_E}, \quad 
  W_{\wt{Q}_E}:= \begin{bmatrix}  W_{1,\wt{Q}_E} & 
  W_{2,\wt{Q}_E} \\ W_{3,\wt{Q}_E} & W_{4,\wt{Q}_E} 
  \end{bmatrix}, 
\end{equation}
where
\begin{equation} \label{W.def}
\begin{split}
  & \ W_{1,Q} = 2\lambda \Re Q \Im Q, \quad W_{2,Q}  = 
  \lambda [ (\Re Q)^2 + 3(\Im Q)^2], \\
  & \ W_{3,Q}  = -\lambda [ 3(\Re Q)^2 + (\Im Q)^2], \quad 
  W_{4,Q} = - W_{1,Q}.
\end{split}
\end{equation}
Now let
\begin{equation}
  \bK_0 := \begin{bmatrix} H_0 - E - W_{3} & W_{1} \\ 
  W_{1} & H_0 - E + W_{2}   \end{bmatrix}.
\end{equation}
We see that $\bK_0$ is self adjoint and 
\[ 
  \bH_0 = J \bK_0, \quad \bH_0^* = - \bK_0 J. 
\]
Moreover, denoting
\[
  \bH_r = \mbox{ linearized operator around excited state }
  \;\; e^{i r_1} \wt{Q}_E (\Gamma_1(r_2,r_3) x)
\]
for $r = (r_1, r_2, r_3) \in \R^3$, we have
\begin{equation} \label{H.sym} 
  \bH_r \left( r*\begin{bmatrix} u \\   v \end{bmatrix}\right) = 
  r* \left( \bH_0 \begin{bmatrix} u \\   v \end{bmatrix}\right), \ 
 \text{where}\ r* \begin{bmatrix} f(x) \\   g(x)\end{bmatrix} = 
 e^{Jr_1}\begin{bmatrix} f(\Gamma_1(r_2,r_3)x) \\  
 g(\Gamma_1(r_2,r_3)x) \end{bmatrix}. 
\end{equation}
Therefore, as above, to study the spectrum of $\bH_r$, it
suffices to study the spectrum of $\bH_0$.
We have the following theorem:
\begin{theorem}[$\bH_0$-Invariant Subspaces -- Resonant Case] 
\label{Spectrum.C.R} 
Assume that $e_0 < 2e_1$. The space 
$\E = L^2(\mathbb{R}^3, \mathbb{C}^2)$ can be decomposed into 
$\bH_0$-invariant subspaces as
\[ 
  \E =  \E_{01} \oplus\E_{02} \oplus \E_1 \oplus \E_2 \oplus 
  \E_+ \oplus \E_- \oplus\E_{c}. 
\]
If $f$ and $g$ belong to different subspaces, then $\wei{J f, g} =0$. These subspaces and their corresponding projections satisfy the following:
\begin{itemize}
\item[\textup{(i)}] 
The subspaces 
\[
  \E_{01} := \text{span}_{\mathbb{C}} \left\{ 
  \Phi_{00} := \frac{\partial}{\partial_E} \vt{\wt{Q}_E}, \;
  \Phi_{01}:= \frac{\partial}{\partial r_1}
  \vt{e^{-i r_1} \wt{Q}_E} \big|_{r_1=0} \right \}
\]
and 
\[
  \E_{02} := \text{span}_{\mathbb{C}} \left\{
  \Phi_{0j}:= \frac{\partial}{\partial r_j}
  \vt{\wt{Q}_E}(\Gamma_1(r_2,r_3) x) \big|_{r_2=r_3=0}
  \;\;\; j=2,3 \right\}. 
\]
Moreover, $ \bH_0 \Phi_{00} = \Phi_{01}, \quad \bH_0 \Phi_{0j} = 0, \quad \forall j = 1,2,3$ and
\[ 
  \wei{J\Phi_{0j}, \Phi_{0k}} = 0, \quad \text{if} \quad (j,k) \notin
  \{(0,1), (1,0), (2,3), (3,2) \}. 
\]
The projections $\bP_0(\bH_{0}) : \E \rightarrow \E_0:=\E_{01} \oplus
\E_{02}$ are defined by  
$\bP_0(\bH_{0})= \bP_{01}(\bH_{0}) + \bP_{02}(\bH_{0})$ with
\begin{flalign*}
 \bP_{01}(\bH_0)f & = \beta_1 \wei{J \Phi_{00}, f}\Phi_{01} - \beta_1 \wei{J\Phi_{01}, f}\Phi_{00}, \\ 
 \bP_{02}(\bH_0)f & =  \beta_2 \wei{J \Phi_{02}, f}\Phi_{03} -\beta_2 \wei{J \Phi_{03}, f}\Phi_{02}, \quad \forall \ f \in \E,
\end{flalign*}
where $\beta_1 := \wei{J\Phi_{00}, \Phi_{01}}^{-1} = O(1)$ and $\beta_2 := \wei{J\Phi_{02}, \Phi_{03}}^{-1} = O(\epsilon^{-2})$.
\item[\textup{(ii)}] 
There exist $\Phi_1 := \begin{bmatrix} -i\phi_1 + \phi_2 \\ \phi_1 + i\phi_2 \end{bmatrix} + O(\epsilon^2), \Phi_2 := \bar{\Phi}_1$ and purely imaginary numbers $\omega_1 := i(\lambda \epsilon^2 + O(\epsilon^3)))$ and $\omega_2 := \bar{\omega}_1$ such that for $j =1,2$, the subspace $\E_j$ is spanned by $\Phi_j$ and 
\[ 
  \bH_0 \Phi_j = \omega_j \Phi_j, \quad \wei{J \Phi_1, \Phi_1} =-4i,
  \quad \wei{J \Phi_1, \Phi_2} =0. 
\] 
The projection $\bP_j(\bH_0) : \E \rightarrow \E_j$ is defined by
\[\bP_1(\bH_0) f = -\frac{1}{4i}\wei{J\Phi_1, f}\Phi_1, \ \ \bP_2(\bH_0) f =  \frac{1}{4i}\wei{J \Phi_2, f}\Phi_2, \quad \forall \ f \in \E.\]
\item[\textup{(iii)}] 
For $j =3,4,5,6$ there exist eigenfunctions $\Phi_j$ with
corresponding eigenvalues $\omega_j$ satisfying the condition 
$\omega_{k+1} = \bar{\omega}_k$, $\Phi_{k+1} = \bar{\Phi}_k$ for
$k=3,5$, and 
$\Im \omega_4 = \kappa =  e_1 - e_0 + O(\epsilon^2) >0$, 
$\Re \omega_4 =  \gamma = \gamma_0\epsilon^4 + O(\epsilon^6)$ 
for some positive order one constant $\gamma_0$. Moreover, we have 
\begin{flalign*}
& \omega_5 = -\omega_3, \quad \wei{J\Phi_5, \Phi_4} =1, \quad  \wei{J\Phi_j,\Phi_j}=0 \quad  \text{for all} \quad j=3,4,5,6.
\end{flalign*}
The subspaces $\E_+ = \text{span}\{\Phi_4, \Phi_5\}$ and $\E_- = \text{span}\{\Phi_3, \Phi_6\}$ with  
the projection 
$\bP_\pm(\bH_0) : \E \rightarrow \E_\pm$ defined by
\begin{flalign*}
  \bP_+(\bH_0) f & = \wei{J\Phi_5, f}\Phi_4 - \wei{J\Phi_4,f}\Phi_5, \\
  \bP_-(\bH_0) f  & = \wei{J\Phi_6,f}\Phi_3 - \wei{J\Phi_3,f}\Phi_6,
  \quad \forall \ f \in \E.
\end{flalign*}
\item[\textup{(iv)}] 
$\E_c = \{g \in \E : \wei{Jf, g} = 0, \quad \forall \ f \in \E_0
  \oplus \E_1 \oplus \E_2 \oplus \E_+ \oplus \E_-\}$. 
Its corresponding projection is 
$\bP_c(\bH_0) = \textup{Id}  - \sum_{j=0}^2 \bP_j(\bH_0) - \bP_+(\bH_0) - \bP_-(\bH_0)$.
\end{itemize}
\end{theorem}
\medskip

\setlength{\unitlength}{1mm}\noindent
\begin{center}
\begin{picture}(120,80)

\put (0,40){\vector(1,0){120}}
\put (120,42){\makebox(0,0)[c]{\scriptsize x}}

\put (60,0){\vector(0,1){80}}
\put (60,82){\makebox(0,0)[c]{\scriptsize y}}

{\color{red}\put (60.15,60){\line(0,1){18}}\put (59.8,60){\line(0,1){18}}}

\multiput(60, 76)(0.5, 0){30}{\circle*{0.1}}\put (75,76){\vector(1,0){1}}
\put(83,76){\makebox(0,0)[c]{\scriptsize$\sigma_c(\bH_0)$}}

\put(64,60){\makebox(0,0)[c]{\scriptsize$|E|$}}

{\color{red}\put (60.15,20){\line(0,-1){20}}\put (59.8,20){\line(0,-1){20}}}
\put(64,20){\makebox(0,0)[c]{\scriptsize$-|E|$}}

\put(60,40){\circle*{1.5}}

\multiput(60, 40)(0.8, -0.4){30}{\circle*{0.1}}\put (84,28){\vector(2,-1){1}}
\put(100,25){\makebox(0,0)[c]{\scriptsize$\E_0 =\{\Phi_{00} =\partial_EQ_{|r=0}, \Phi_{0j}=\partial_{r_j}Q_{|r=0}, \ j = 1,2,3\}$}}

\put(60,44){\circle*{1.5}}
\put(73,44){\makebox(0,0)[c]{\scriptsize$(\omega_1 =O(\ep^2), \Phi_1)$}}

\put(60,36){\circle*{1.5}}
\put(45,36){\makebox(0,0)[c]{\scriptsize$(\omega_2 =\bar{\omega}_1, \Phi_2 =\bar{\Phi}_1)$}}

\put(40,72){\makebox(0,0)[c]{\scriptsize$\E_{+}= \{\Phi_4, \Phi_5\}$}}

\put(62,66){\circle*{1.5}}
\put(80,66){\makebox(0,0)[c]{\scriptsize$(\omega_4, \Phi_4),\ \Re(\omega_4) = O(\ep^4)$}}

\put(58,66){\circle*{1.5}}
\put(47,66){\makebox(0,0)[c]{\scriptsize$(\omega_5, \Phi_5)$}}

\put(62,14){\circle*{1.5}}
\put(77,14){\makebox(0,0)[c]{\scriptsize$(\omega_3 =\bar{\omega}_4, \Phi_3= \bar{\Phi}_4)$}}

\put(58,14){\circle*{1.5}}
\put(43,14){\makebox(0,0)[c]{\scriptsize$(\omega_6=\bar{\omega}_5, \Phi_6 = \bar{\Phi}_5)$}}

\put(40,8){\makebox(0,0)[c]{\scriptsize$\E_{-}= \{\Phi_3, \Phi_6\}$}}

\multiput(60, 4)(0.5, 0){30}{\circle*{0.1}}\put (75,4){\vector(1,0){1}}
\put(82,4){\makebox(0,0)[c]{\scriptsize$\sigma_c(\bH_0)$}}

\end{picture}

Figure 3: Spectrum of $\bH_0$ in the resonant case.
\end{center}
\begin{proof} 
From Lemma \ref{Non-emb-eig} and Lemma \ref{outside}, we see that
$\bH_0$ has no eigenvalues in 
$\Sigma_c \cup \{z : dist(z,\Sigma_p) \geq \ep \}$. 
Therefore, we shall only need to look for eigenvalues of $\bH_0$ in
$\{z : \text{dist}(z,\Sigma_p) \leq \ep \ \text{and}\ z \notin
\Sigma_c \}$. 
First of all, note that $\bH_0$ is a small perturbation of the
operator $JH$ (recall that $H= H_0-E$) whose discrete spectrum 
$\sigma_d(JH) =\{ \pm i(e_0-E), \pm i(e_1-E)\}$, 
and whose eigenfunctions are
$\Phi_j^{\pm} := \begin{bmatrix} \phi_j \\ \pm i\phi_j \end{bmatrix}$, 
for $j = 0,1,2,3$. The eigenvalues $\pm i(e_0-E)$ are simple 
and the eigenvalues $\pm i(e_1-E) = O(\epsilon^2)$ have 
multiplicity $3$. Moreover, the continuous spectrum of 
$JH$ is $\Sigma_c :=\sigma_c(J(H_0-E) =\{i e : |e| \geq -E\}$. 

From standard perturbation theory, the dimension of eigenspaces 
of $\bH_0$ for eigenvalues near $0$ totals $6$, and the 
corresponding eigenfunctions are perturbations of linear 
combinations of $\Phi_{j}^{\pm}, j =1,2,3$. 
On the other hand, by the resonance condition, 
$|e_0 -E| = e_1 -e_0 +O(\epsilon^2) > -E$. So finding the 
eigenvalues of $\bH_0$ bifurcating from $\pm i (e_1-E)$ requires 
careful estimates of resolvent operators. 
In general, we need to solve the problem
\begin{equation} \label{H.eig}
  \bH_0 \begin{bmatrix} u \\ v \end{bmatrix} = 
  \tau \begin{bmatrix} u \\ v \end{bmatrix} 
\end{equation}
for some functions $u, v$ and some eigenvalue $\tau$ near zero 
and near $\pm i(e_1 -e_0)$. 

We shall first find the eigenvalues $\tau$ near zero. 
Let's write $u = a\cdot \phi + h_1, v = b \cdot \phi + h_2$ where 
$a = (a_1, \cdots, a_3)$ and $b = (b_1, \cdots, b_3)$ are of order one and $h_1, h_2 \in \bV^\perp$. Then, we have
\begin{equation} \label{h12.eqn}
  \left \{ \begin{array}{ll}
  (H_0- E) h_2 + W_1 u + W_2 v & = \tau u + (E-e_1) b \cdot \phi,\\
  (H_0- E)h_1  -W_3 u - W_4 v & = -\tau v + (E-e_1) a \cdot \phi.
\end{array} \right.
\end{equation}
Applying the projection $\bP_1^\perp$, we get
\begin{equation} \label{h12.sol} 
  h_1 = -(H_0-e_1)^{-1}\bP_1^\perp (\tau h_2 - W_3 u - W_4 v), 
  \quad h_2 = (H_0-e_1)^{-1}\bP_1^\perp (\tau h_1 - W_1u - W_2 v).
\end{equation}
Taking the inner product of the equations \eqref{h12.eqn} with $\phi_j$, for $j=1,2,3$, we get
\begin{equation} \label{e.eqn} 
  (\phi_j, W_1 u + W_2 v) = \tau a_j + (E-e_1) b_j, \quad 
  (\phi_j, W_3u + W_4 v)= \tau b_j + (e_1-E) a_j. 
\end{equation}
From \eqref{h12.sol} and \eqref{e.eqn}, we see that $h_1, h_2, \tau = O(\epsilon^2)$. We now write $\tau = \frac{\lambda \epsilon^2}{2} \tau_1 + O(\epsilon^3)$. Then, from \eqref{e.eqn} we have
\begin{equation*}
\begin{bmatrix} 
0 & 1 & 0 & 1 & 0 & 0 \\
1 & 0 & 0 & 0 & 3 & 0 \\
0 & 0 & 0 & 0 & 0 & 0 \\
-3 & 0 & 0 & 0 & -1 & 0 \\
0 & -1 & 0 & -1 & 0 & 0\\
0 & 0 & 0 & 0 & 0 & 0 
\end{bmatrix} \begin{bmatrix} a_1 \\ a_2 \\ a_3 \\ b_1 \\ b_2 \\ b_3
\end{bmatrix} = \tau_1 \begin{bmatrix} a_1 \\ a_2 \\ a_3 \\ b_1 \\ b_2 \\ b_3
\end{bmatrix}.
\end{equation*}
This implies that $\tau_1$ solves the equation $ \tau_1^4(\tau_1^2 + 4) = 0$. So, the eigenvalue problem \eqref{H.eig} has two purely imaginary eigenvalues $\tau = \omega_1 := \lambda \epsilon^2 i + O(\epsilon^3)$ and $\tau = \omega_2 := \bar{\omega}_1$. Respectively, the eigenvectors of $\omega_1$ and $\omega_2$ are $\Phi_1, \bar{\Phi}_1$ with
\[ 
  \Phi_1 = \begin{bmatrix} -i \phi_1 + \phi_2 \\  
  \phi_1 + i\phi_2 \end{bmatrix} + O(\epsilon^2). 
\]
Moreover, as in the previous section, let
\begin{equation}
\begin{split}
  & \Phi_{00} = \frac{\partial}{\partial E} \vt{\wt{Q_E}} 
  = \partial_E \wt{\rho}(\epsilon) \begin{bmatrix} \phi_1 \\ 
  \phi_2 \end{bmatrix} + O(\epsilon^2), \\
  & \Phi_{01} = \frac{\partial}{\partial r_1} 
  \vt{e^{-i r_1} \wt{Q}_E} \big|_{r_1=0} = \begin{bmatrix} 
  \Im \wt{Q} \\ -\Re \wt{Q}  \end{bmatrix} = \wt{\rho}(\epsilon) 
  \begin{bmatrix} \phi_2 \\ -\phi_1 \end{bmatrix} + O(\epsilon^3), \\
  & \Phi_{02} = \frac{\partial}{\partial r_2}
  \vt{\wt{Q}_E}(\Gamma_1(r_2,r_3) x) \big|_{r_2=r_3=0}
  =  \wt{\rho}(\epsilon) \begin{bmatrix} 0 \\ \phi_3 \end{bmatrix} 
  + O(\epsilon^3), \\
  & \Phi_{03} = \frac{\partial}{\partial r_3} 
  \vt{\wt{Q}_E}(\Gamma_1(r_2,r_3) x) \big|_{r_2=r_3=0}
  =  \wt{\rho}(\epsilon) \begin{bmatrix}  \phi_3 \\ 0 
  \end{bmatrix} + O(\epsilon^3).
\end{split}
\end{equation}
We see that $\bH_0$ has three zero - eigenvectors and one zero-generalized eigenvector. Precisely, we have
\[ \bH_0 \Phi_{00} = \Phi_{01}, \quad  \bH_0 \Phi_{0j} = 0, \ \forall \ j = 1, 2,3. \]

Next, we shall look for the eigenvalues $\tau$ of $\bH_0$ near $\pm
i(e_0 - E)$, i.e. we shall solve $(JH -\tau +W)\Phi =0$. We shall
first solve this equation for $|\tau -i(e_0-E)| \leq \ep$. From Lemma
\ref{Non-emb-eig}, we may assume that $\Re \tau \not=0$. We write
$\Phi = a\Phi_0^+ + \eta$, for some $a \in \mathbb{C}$,
$\eta \perp \Phi_0^+$.  Then we have 
\begin{equation} \label{he.tau.eta}
\left \{
\begin{array}{ll}
& a\tau  = ai(e_0-E) + \frac{1}{2}\wei{\Phi_0^+, W(a\Phi_0^+ +\eta)}, \\
& (JH -\tau) \eta = -\bP^\perp W(a\Phi_0^+ +\eta).
\end{array} \right.
\end{equation}
Here $\bP^\perp =\bP_d + \bP_c$,
\[\bP_d f :=  \frac{1}{2}\wei{\Phi_0^-, f}\Phi_0^- + \sum_{j=1}^3 \frac{1}{2}\wei{\Phi_j^{\pm}, f}\Phi_j^{\pm}, \quad \bP_c f := \bP_c(JH)f,   \quad f \in \E. \]
Recall that $R_0(z) = (J H - z)^{-1} =R_{01}(z) + R_{02}(z)$ where,
\[ R_{01}(z):= \frac{1}{2}\begin{bmatrix} i & -1 \\ 1 & i\end{bmatrix}(H -iz)^{-1}, \quad R_{02}(z):= \frac{1}{2}\begin{bmatrix} -i & -1 \\ 1 & -i\end{bmatrix}(H + iz)^{-1}.\]
Therefore, $R_0(z)$ is well defined for $\Re z \not=0$. From the second equation of \eqref{he.tau.eta}, we get
\begin{equation} \label{eta.spectrum.eqn}
\eta = -R_0(\tau) \bP^\perp [W(a\Phi_0^+ + \eta)].
\end{equation}
For $j=1,2,3$, from \eqref{W.def} and by the symmetry properties of
$\wt{Q}$, we have 
\begin{equation} \label{in.com}
\begin{split}
\wei{\Phi_j^{+}, W\Phi_0^+} & = (\phi_0\phi_j, W_1 + iW_2) - i(\phi_0\phi_j,W_3 + iW_4) \\
& =4i\lambda(\phi_0\phi_j, |\wt{Q}|^2) =0, \\
\wei{\Phi_j^{-}, W\Phi_0^+} & = (\phi_0,\phi_j, W_1 + iW_2) + i(\phi_0\phi_j,W_3 + iW_4)\\
&  = -2i\lambda (\phi_0\phi_j, \wt{Q}^2)=0. 
\end{split}
\end{equation}
 So,
\[ \bP^\perp [W\Phi_0^+] = \frac{1}{2}\wei{\Phi_0^-, W\Phi_0^+}\Phi_0^- + \Pc W\Phi_0^+.\]
Therefore, we obtain the equation for $\eta$
\begin{equation} \label{eta.exp.sp}
\begin{split}
\eta_\tau & = \frac{1}{2[i(e_0-E) +\tau]}[a\wei{\Phi_0^-, W\Phi_0^+} + \wei{\Phi_0^-, W\eta}]\Phi_0^- \\
\ & \quad - \sum_{j=1}^3 \frac{1}{2[\pm i(e_1-E)-\tau]}\wei{\Phi_j^\pm, W\eta} \Phi_j^\pm - aR_0(\tau)\Pc W\Phi_0^+  - R_0(\tau)\bP_c W\eta.
\end{split}
\end{equation}
Now suppose that $a =0$. Since $\Re(\tau)\not=0$ and $\tau$ is close
to $ i(e_0-E)$, from \eqref{eta.exp.sp} and Lemma \ref{R0.sp.est}, we
see that 
$\norm{\eta}_{L^2\loc} \leq C\ep^2\norm{\eta}_{L^2\loc}$. 
So, we get a contradiction because $\ep$ is small. 
So, without loss of generality, we may assume that $a=1$. Then let's define $\tau_0 : = i(e_0-E) + \frac{1}{2}(\Phi_0^+, W \Phi_0^+)$. Since
\[ 
  \wei{\Phi_0^+, W \Phi_0^+}  = (\phi_0^2, W_1+iW_2) -i(\phi_0^2,
  W_3-iW_1) = 4i(\phi_0^2, |\wt{Q}|^2),
\]
we see that $\tau_0$ is purely imaginary and $\Im \tau_0 = e_0-e_1 + \lambda \ep^2 +O(\epsilon^4)$.  Moreover, let
\begin{equation}\label{tau1.def} 
\begin{split}
  \tau_1 & :=  \frac{1}{4[i(e_0-E)+\tau_0]}|\wei{\Phi_0^-, W\Phi_0^+}|^2 - 4i\lambda^2 (\phi_0|Q|^2, \frac{1}{H +i\tau_0 }\bP_c(H_0)\phi_0|Q|^2) \\
 \ & \quad -\lambda^2 i (\phi_0Q^2, \frac{1}{H -i\tau_0 -i0}\bP_c(H_0)\phi_0Q^2).
\end{split} 
\end{equation}
Then we have $|\tau_1| \les \epsilon^4$. Moreover, from \eqref{FGR}, and since $(H+i\tau_0)^{-1}\bP_c$ is regular, we have $\Re (\tau_1) \geq \lambda^2\lambda_0 \epsilon^4$. Here, $\lambda$ and $\lambda_0$ are respectively defined in \eqref{NSE} and \eqref{FGR}. From \eqref{he.tau.eta} and \eqref{eta.spectrum.eqn}, we get the equation for $\tau$
\begin{equation} \label{tau.eqn.sp}
\tau = \tau_0 + \tau_1 - [\tau_1 + \frac{1}{2}\wei{\Phi_0^+, WR_0(\tau)\bP^\perp W\Phi_0^+}] - \frac{1}{2}\wei{\Phi_0^+, WR_0(\tau)\bP^\perp W\eta}.
\end{equation}
From \eqref{eta.exp.sp} and Lemma \ref{R0.sp.est} and by applying the contraction mapping theorem, for each $\tau$ in $\mathbf{D}_1$, we can find the solution $\eta_\tau$ in $L^2_{-s}$ of \eqref{eta.exp.sp} such that $\norm{\eta_\tau}_{L^2_{-s}} \les \epsilon^2$. Moreover, for all $\tau, \tau' \in \mathbf{D}_1$, from \eqref{eta.exp.sp} and Lemma \ref{R0.sp.est}, we also have
\[ \norm{\eta_{\tau} -\eta_{\tau'}}_{L^2_{-s}} \les \epsilon^2[\max\{\Re \tau, \Re\tau'\}]^{-1/2}|\tau-\tau'| + \epsilon^2 \norm{\eta_{\tau} -\eta_{\tau'}}_{L^2_{-s}}.\]
Therefore, there exists a constant $C>0$ such that
\begin{equation} \label{eta.contraction} \norm{\eta_{\tau} -\eta_{\tau'}}_{L^2_{-s}} \leq C\epsilon^2[\max\{\Re \tau, \Re\tau'\}]^{-1/2}|\tau-\tau'|, \quad \forall \ \tau, \tau' \in \mathbf{D}_1. \end{equation}
Next, let
\begin{equation} \label{D.def} \mathbf{D}: = \{z \in \mathbb{C}: |z-(\tau_0+\tau_1)| \leq \epsilon^5\}.\end{equation}
Note that $\mathbf{D} \subset \mathbf{D}_1$. Also, let $f(\tau)$ be a function which is defined by the right hand side of \eqref{tau.eqn.sp}. We shall show that $f$ maps $\mathbf{D}$ into $\mathbf{D}$ and has a fixed point in $\mathbf{D}$. Note that 
\begin{equation} \label{exp.Pperp} \begin{split}
& \wei{\Phi_0^+, WR_0(\tau)\bP^\perp W\Phi_0^+} = \frac{1}{2}\wei{\Phi_0^-, W\Phi_0^+} \wei{\Phi_0^+, WR_0(\tau)\Phi_0^-} + \wei{\Phi_0^+, WR_0(\tau)\Pc W\Phi_0^+} \\
& \quad \quad \quad = -\frac{1}{2[i(e_0-E) + \tau]}\wei{\Phi_0^-, W\Phi_0^+}\wei{\Phi_0^+, W\Phi_0^-} + \wei{\Phi_0^+, WR_0(\tau)\Pc W\Phi_0^+} \\
& \quad \quad \quad = -\frac{1}{2[i(e_0-E) + \tau]}|\wei{\Phi_0^-, W\Phi_0^+}|^2 + \wei{\Phi_0^+, WR_0(\tau)\Pc W\Phi_0^+},\\
& \quad \quad \quad = -\frac{1}{2[i(e_0-E) + \tau]}|\wei{\Phi_0^-, W\Phi_0^+}|^2 + 2i\lambda^2(\phi_0Q^2, (H-i\tau)^{-1}\bP_c(H) \phi_0Q^2)\\
& \quad \quad \quad \quad \quad \quad \quad \quad +4i\lambda^2(\phi_0|Q|^2, (H+i\tau)^{-1}\bP_c(H)\phi_0|Q|^2).
\end{split}
\end{equation}
Here, we have used the fact that $R_0\Pc = R_{01}\Pc + R_{02}\Pc$ with
\[ R_{01}(z)\Pc:= \frac{1}{2}\begin{bmatrix} i & -1 \\ 1 & i \end{bmatrix}(H-iz)^{-2}\bP_c(H), \quad R_{02}(z)\Pc:= \frac{1}{2}\begin{bmatrix} -i & -1 \\ 1 & -i\end{bmatrix}(H+iz)^{-1}\bP_c(H)\]
and
\begin{equation} \label{R0z.in}
\begin{split}
\wei{\Phi_0^+,  WR_{01}(\tau)\Pc W\Phi_0^+} & = 2i\lambda^2(\phi_0Q^2, (H-i\tau)^{-1}\bP_c(H) \phi_0Q^2), \\
\wei{\Phi_0^+,  WR_{02}(\tau)\Pc W\Phi_0^+} & = 8i\lambda^2(\phi_0|Q|^2, (H+i\tau)^{-1}\bP_c(H)\phi_0|Q|^2).
\end{split}
\end{equation}
From \eqref{tau1.def}, \eqref{exp.Pperp}, Lemma \ref{R0.sp.est} and $|\tau -\tau_0| \les \epsilon^4$ if $\tau \in \mathbf{D}$, we obtain
\[ |\tau_1 + \frac{1}{2}\wei{\Phi_0^+,  WR_0(\tau)\bP^\perp W\Phi_0^+}| \les \epsilon^6, \ \forall \ \tau \in \mathbf{D}. \]
On the other hand,
\begin{equation*}
\begin{split}
&  \wei{\Phi_0^+, WR_0(\tau)\bP^\perp W\eta} = -\frac{1}{2[i(e_0-E) +\tau]}\wei{\Phi_0^-, W\eta}\wei{\Phi_0^+, W\Phi_0^-} \\
& \quad \quad \quad + \sum_{j=1}^3 \frac{1}{2[\pm i(e_1-E) -\tau]}\wei{\Phi_0^\pm,W\eta}\wei{\Phi_0^+,W\Phi_j^\pm} + \wei{\Phi_0^+, WR_0(\tau)\bP_c W\eta}.
\end{split} \end{equation*}
So,
\[ |\wei{\Phi_0^+, WR_0(\tau)\bP^\perp W\eta}| \les \ep^6, \ \forall \ \tau \in \mathbf{D}.\]
Therefore, $|f(\tau) - (\tau_0 + \tau_1)| \les \epsilon^6 \leq \epsilon^5$ for all $\tau \in \mathbf{D}$. So, $f(\tau) \in \mathbf{D}$ for all $\tau \in \mathbf{D}$. From \eqref{eta.contraction}, \eqref{exp.Pperp} and Lemma \ref{R0.sp.est} below, we can show that there exists a constant $C>0$ independent of $\epsilon$ such that $|f(\tau)-f(\tau')| \leq C\epsilon^2|\tau -\tau'|$, for all $\tau, \tau'$ in $\mathbf{D}$. So, the map $f : \mathbf{D} \rightarrow \mathbf{D}$ is a contraction when $\epsilon$ is sufficiently small. Therefore, there exists unique $\tau_* \in \mathbf{D}$ such that $\tau_* = f(\tau_*)$. Moreover, we have
\begin{equation} \label{tau*size}
 \Im \tau_* = \Im \tau_0 + O(\epsilon^2) = e_0 - e_1 + O(\epsilon^2), \  \lambda^2\lambda_0 \epsilon^4 \leq \Re \tau_* = \Re \tau_1 +O(\epsilon^6) \leq C\epsilon^4. \end{equation}
Next, we shall show that $\tau_*$ is the unique fixed point of $f$ in $\mathbf{D}_1$, where $\mathbf{D}_1$ is defined in Lemma \ref{R0.sp.est} (note that $\mathbf{D} \subset \mathbf{D}_1$). Suppose that there is $\tau' \in \mathbf{D}_1$ such that $f(\tau') = \tau'$. Using \eqref{eta.contraction} and \eqref{tau*size}, we obtain
\[ |\tau_* -\tau'| = |f(\tau_*) - f(\tau')| \leq C \ep^2|\tau_ - \tau'| \leq \frac{1}{2}|\tau_* - \tau'|.\]
Therefore, $\tau' = \tau_*$. So, $\tau_*$ is the unique eigenvalue of $\bH_0$ in $\mathbf{D}_1$. In summary, let $h_3:=\eta_{\tau_*}$, $\Phi_3 = \Phi_0^{+} + h_3$, $\omega_3 = \tau_*$, $\Phi_4 =\bar{\Phi}_3$ and $\omega_4 = \bar{\omega}_3$. We have $\lambda_0 \lambda^2 \leq \Re \omega_j \les \epsilon^4$, $\Im \omega_{j} = (-1)^{j-1}(e_0-e_1) + O(\ep^2)$ and
\[ \bH_0 \Phi_j = \omega_j\Phi_j, \ \ \forall \ j = 3,4.\]
Moreover, from \eqref{eta.exp.sp}, we get
\begin{equation} \label{h3.def}
\begin{split}
h_3 & = \frac{1}{2[i(e_0-E +\omega_3]}[\wei{\Phi_0^-, W\Phi_0^+} +\wei{\Phi_0^-, Wh_3}]\Phi_0^- \\
& - \sum_{j=1}^3 \frac{1}{\pm i(e_1-E) -\omega_3}\wei{\Phi_j^\pm, Wh_3}\Phi_j^\pm -R_0(\omega_3)\bP_cW\Phi_0^+ - R_0(\omega_3)\bP_cWh_3.
\end{split}
\end{equation}
Moreover, from Lemma \ref{Lp.eta.sp} below, we see that $\norm{h_3}_{L^2} \leq C$.
Similarly, solving the bifurcation equation $(JH +W)[\Phi_0^- + \wt{\eta}] = z(\Phi_0^- + \wt{\eta})$, we also obtain two other eigenvalues and eigenfunctions $\omega_5, \omega_6$ and $\Phi_5 = \Phi_0^- + h_5, \Phi_6 = \Phi_0^+ + h_6$ with $\omega_6 =\bar{\omega}_5, \Phi_6 =\bar{\Phi}_5$, $\bH_0 \Phi_j = \omega_j \Phi_j$, $\norm{h_j}_{L^2} \leq C$ for $j=5,6$. In particular, $h_5$ satisfies
\begin{equation} \label{h5.def}
\begin{split}
h_5 & = \frac{1}{2[-i(e_0-E +\omega_5]}[\wei{\Phi_0^+, W\Phi_0^-} +\wei{\Phi_0^+, Wh_3}]\Phi_0^+ \\
& - \sum_{j=1}^3 \frac{1}{\pm i(e_1-E) -\omega_5}\wei{\Phi_j^\pm, Wh_3}\Phi_j^\pm -R_0(\omega_5)\bP_cW\Phi_0^- - R_0(\omega_5)\bP_cWh_5.
\end{split}
\end{equation}
Also,
\begin{equation*}
\begin{split}
\Re \omega_5 & = -\lambda^2 \Im (\phi_0\bar{Q}^2, \frac{1}{H_0 +e_0 -2e_1 +O(\ep^2) -i0}\bP_c \phi_0\bar{Q}^2) + O(\ep^6), \\
\Im \omega_5 & = -(e_0-e_1) + O(\ep^2).
\end{split}
\end{equation*}
In particular, we have $\Re \omega_j = O(\ep^4)$ and $\Re \omega_j \leq -\lambda_0\lambda_0 \ep^4$ for $j=5,6$. By the uniqueness properties of these fixed points $\omega_3, \omega_4, \cdots, \omega_6$, we see that they are all of eigenvalues of $\bH_0$ in the $\ep$-neighborhood of $\pm i e_{01}$. Therefore, we have found all of the eigenvalues of $\bH_0$. \\


Next, we shall prove the orthogonality conditions from which the formulas of the projections follow. Firstly, from the symmetry properties of $\wt{Q}_E$, it follows that $\wei{Jf,g}=0$ if $f,g$ are in different spaces of $\E_{01}$ and $\E_{02}$. Secondly, we claim that for any two pairs of eigenvectors, eigenvalues $(u_1, \lambda_1), (u_2, \lambda_2) $ of $\bH$  with $\bar{\lambda}_1 + \lambda_2 \not=0$, then $\wei{Ju_1,u_2} =0$. In fact, since $\bH = J\bK$ and $\bK$ is a self-adjoint operator, we get 
\begin{equation} \label{com.abs}
 \lambda_2 \wei{Ju_1,u_2} = \wei{Ju_1, \bH u_2} = \wei{\bH^* Ju_1, u_2} = \wei{\bK u_1, u_2} =- \bar{\lambda}_1\wei{Ju_1, u_2}.  \end{equation}
Then, we obtain $(\bar{\lambda}_1 +\lambda_2)\wei{Ju_1, u_2} =0$. So, $\wei{Ju_1, u_2}= 0$ since $\bar{\lambda}_1 + \lambda_2 \not=0$.  Therefore, $\wei{Jf,g} =0$ if $f$ and $g$ are in different spaces of $\E_1, \E_2, \E_+, \E_-$. Moreover, $\wei{J\Phi_j,\Phi_j} =0$ for all $j=3,4,5,6$. On the other hand, for $u_1 \in \E_0$, we have $(\bH^*)^2 J u_1 =(\bK J)(\bK J) Ju_1 = J \bH^2u_1 =0$. Therefore, for all $u_2$ such that $\bH u_2 = \lambda u_2$ with $\lambda \not=0$, we get
\begin{equation*} 
 \lambda^2\wei{Ju_1,u_2} = \wei{Ju_1, \bH^2 u_2} = \wei{(\bH^*)^2 u_1, u_2} =0. \end{equation*}
This proves that $\wei{Jf,g} =0$ if $f \in \E_0$ and $g$ is in one of $\E_1, \E_2, \E_+, \E_-$. So, we have proved all of the orthogonality conditions. Moreover, since $J \Phi_0^+ = i\Phi_0^+$ and Lemma \ref{Lp.eta.sp} below, we obtain
\begin{flalign*}
 \wei{J\Phi_5, \Phi_4}  & = \wei{J (\Phi_0^+ + h_5), \Phi_0^+ + h_4} = 2i  + O(\ep^2) + \wei{Jh_5, h_4}\\
& =  2i + O(\ep^2) \not=0.
\end{flalign*}
Moreover, it follows from this and \eqref{com.abs} that
\[ \omega_5 = -\bar{\omega}_4= -\omega_3. \]
Finally, we shall complete the proof of Theorem \ref{Spectrum.R} by the following two lemmas which show that all of $\omega_3, \omega_4, \cdots \omega_6$ are simple and $\norm{h_j}_{L^2} \leq C$ for $j =3,4,\cdots, 6$:
\end{proof}
\begin{lemma} \label{Lp.eta.sp} For $1 \leq p \leq \infty$, we have 
\[ \norm{h_j}_{L^p} \leq C_p [\ep^2 + \ep^{6-\frac{12}{p}}],  \quad \norm{h_j}_{H^1} \leq C, \quad |\wei{Jh_5, h_4}| \leq C\ep^4, \ \forall j = 3,4,5,5. \]
\end{lemma}
\begin{proof} To prove the first inequality of Lemma \ref{Lp.eta.sp}, we only need to prove it for $j=3$. From \eqref{h3.def}, we only need to show that 
\begin{equation} \label{L2.eta.sp}
\norm{R_0(\omega_3)\bP_cW\Phi_0^+}_{L^p}, \ \norm{R_0(\omega_3)\bP_cW\eta}_{L^p} \leq C_p[\ep^2 + \ep^{6-\frac{12}{p}}].
\end{equation}
Since $R_0(\tau) = R_{01}(\tau) + R_{02}(\tau)$ and $R_{02}(\omega_3)$ is regular and $R_{01}(\omega_3) \sim (H-i\omega_3)^{-1}$, we shall only need to prove \eqref{L2.eta.sp} with $R_0(\omega_3)$ replaced by $(H-i\omega_3)^{-1}$. Now, we follow the argument in \cite{NTP}. We write 
\[ H-i\omega_3 = H_0 -E -i\omega_3 = -\Delta - \nu^2 + V,\]
where $\nu^2 = E + i\omega_3$ and $\Im \nu >0$ and is of order $\epsilon^4$. By resolvent expansion, we have
\begin{equation} \label{L2.nu.sp}
(H-i\omega_3)^{-1}\varphi = (-\Delta - \nu^2)^{-1}\varphi +(-\Delta - \nu^2)^{-1} V(H-i\omega_3)^{-1}\varphi.  \end{equation}
Because the resolvent $(-\Delta - \nu^2)^{-1}$ has the kernel $K(x) := (4\pi|x|)^{-1} \exp(i\nu|x|)$, we have
\begin{flalign*} \norm{(-\Delta - \nu^2)^{-1}\varphi}_{L^p} & \les \norm{K*\varphi}_{L^p} \les [\norm{K}_{L^p(B_1^c)} + \norm{K}_{L^2(B_1)}]\norm{\varphi}_{L^1 \cap L^2} \\
\ & \les [1 + \ep^{4 - 6/p}]\norm{\varphi}_{L^1 \cap L^2}.
\end{flalign*}
On the other hand, since $V$ decays sufficiently fast, we have
\[ \norm{V(H-i\omega_3)^{-1}\varphi}_{L^1 \cap L^2} \les \norm{(H-i\omega_3)^{-1}\varphi}_{L^2\loc} \les \norm{\wei{x}^{s}\varphi}_{L^2}. \]
Then, it follows from \eqref{L2.nu.sp} that
\[ \norm{(H-i\omega_3)^{-1}\varphi}_{L^2} \les [1 + \ep^{4 - 6/p}][\norm{\phi}_{L^1} + \norm{\wei{x}^{s}\varphi}_{L^2}]. \]
Using this estimate, we get
\begin{flalign*}
\norm{R_0(\omega_3)\bP_cW\Phi_0^+}_{L^p} +\norm{R_0(\omega_3)\bP_cW\eta}_{L^p} \leq C_p \ep^2[1 + \ep^{4 -\frac{12}{p}}].
\end{flalign*}
So, we obtain the first inequality of Lemma \ref{Lp.eta.sp}. Similarly, to prove the second inequality, we only need to show that for some localized function $\varphi$
\[ \norm{\nabla v}_{L^2} \leq C, \quad v: = (H-i\omega_3)^{-1} \varphi. \]
This follows directly by multiplying the equation $(H-i\omega_3) v= \varphi$ by $\bar{v}$ and integrating over $\mathbb{R}^3$. 

Finally, we prove the last inequality of Lemma \ref{L2.eta.sp}. Again, note that $h_3$ and $h_5$ are of the form
\[ h_3 = \text{ok} + A_1(H-i\omega_3)^{-1}\varphi, \quad h_5 = \text{ok} + A_2(H+i\omega_5)^{-1}\varphi_*. \]
Here, $\text{ok}$ = terms which are localized and of order $\ep^2$, $\varphi, \varphi_*$ are some localized functions of order $\ep^2$ and $A_1, A_2$ are some constant matrices of order one. So, we get
\begin{flalign*}
 |\wei{Jh_5, \bar{h}_3}| & \leq C[\ep^4 + |((H+i\omega_5)^{-1}\varphi_*, (H+i\bar{\omega}_3)^{-1}\bar{\varphi})|] \\
 & \leq C[\ep^4 + |((H-i\omega_3)^{-1} (H+i\omega_5)^{-1}\varphi_*, \bar{\varphi})|] \\
 & \leq C\ep^4.
\end{flalign*}
Here, in the last step, we used the fact that for $\al_1, \al_2 \in \mathbb{C}$ such that $\Im \al_j >0$ and $|\Re \al_j| \in [a_1, a_2] \subset (0, \infty)$ with $j =1,2$, we have
\begin{equation} \label{bo.H} 
\norm{(H-\al_1)^{-1}(H-\al_2)^{-1}P_c(H)g}_{L^2_{-s}}  \leq C\norm{g}_{L^2_{s}}, \ \forall \ g \in L^2_s. \end{equation}
Here, the constant $C>0$ is independent of $\al_1$ and $\al_2$.  One can prove \eqref{bo.H} by using Mourre estimates and the argument as in \cite{TY1}, where the authors proved similar estimates for linearized operators and $\al_1 = \al_2$. For a different approach, one can see \cite{C2}.
\end{proof}

\begin{lemma} \label{sim.ome} The eigenvalues $\omega_j$ defined in the proof of Theorem \ref{Spectrum.C.R} are simple for $j=3,4,5,6$.
\end{lemma}
\begin{proof}
It suffices that we only prove this lemma for $j=3$. Suppose by contradiction that there is $\wt{\Phi}$ such that $[\bH - \omega_3] \wt{\Phi} = \Phi_0^+ +h_3$. Note that $h_3 = \eta = \eta(\omega_3)$.  We write $\wt{\Phi} = c\Phi_0^+ +h$ where $h \in \E, h \perp \Phi_0^+$ and $c$ is in $\mathbb{C}$. Then, we have
\[ c[i(e_0-E) -\omega_3]\Phi_0^+ + W [c\Phi_0^+ + h] + (JH -\omega_3)h = \Phi_0^+ + \eta. \]
Equivalently,
\begin{equation}\label{sys.gen}
\left \{\begin{array}{ll} 
& c[i(e_0-E) +\frac{1}{2}\wei{\Phi_0^+, W\Phi_0^+} -\omega_3] +  \frac{1}{2}\wei{\Phi_0^+, Wh} =1, \\
& (JH -\omega_3)h  = -  c\bP^{\perp} W\Phi_0^+ - \bP^\perp Wh + \eta. \end{array} \right.
\end{equation}
Let's now define $h^* = h-c\eta$. From \eqref{he.tau.eta} and \eqref{sys.gen}, we see that $h^*$ solves the equation
\[ h^* = R_0(\omega_3)[\eta - \bP^\perp Wh^*]. \]
From \eqref{eta.exp.sp}, Lemma \ref{R0.sp.est} and \eqref{bo.H}, we have
\[ 
  \norm{R_0(\omega_3)\eta}_{L^2_{-s}} \leq C[\ep^2 +
  \norm{R_0^2(\omega_3)\bP_cW\Phi_0^+}_{L^2_{-s}} +
  \norm{R_0^2(\omega_3)\bP_cW\eta}_{L^2_{-s}}] \leq C\ep^2. 
\]
Therefore, we get $\norm{h^*}_{L^2_{-s}} \leq C\ep^2$. Now, from the first equation of \eqref{sys.gen}, we get
\[c[i(e_0-E) +\frac{1}{2}\wei{\Phi_0^+, W(\Phi_0^++\eta)} -\omega_3] = 1 - \frac{1}{2}\wei{\Phi_0^+, Wh^*}. \]
From \eqref{he.tau.eta}, we see that $i(e_0-E) +\frac{1}{2}\wei{\Phi_0^+, W[\Phi_0^+ +\eta]} -\omega_3 =0$. Therefore, we get $1 - \frac{1}{2}\wei{\Phi_0^+, Wh^*} =0$. This is a contradiction since $|(\Phi_0^+, Wh^*)| \leq C\ep^4$. So, $\omega_3$ is simple and the lemma follows.
\end{proof}
As a corollary of Theorem \ref{Spectrum.C.R}, we obtain the same
spectral properties around symmetry transformed ground states:
\begin{corollary} \label{Spectrum.C.R-cor} For 
$r=(r_1,r_2,r_3) \in \mathbb{R}^3,$ and for $j=0,1,2,3$, let 
$\Phi_{0j}^r : = r* \Phi_{0j} = e^{-Jr_1} \Phi_{0j} \circ \Gamma_1(r_2,r_3)$. Similarly, for $j = 1,2,3,4,5,6$, let $\Phi_j^r := e^{-Jr_1}\Phi_j \circ \Gamma_1(r_2,r_3)$. Also, let $\E_j^r, \E^r_-, \E_+^r, \E_c^r$ be exactly as in Theorem \ref{Spectrum.C.R}, for $j=0,1,2$.
Then $\E_j^r$ are invariant under $\bH_r$, 
$\E = \oplus_{j=0}^2 \E_j^r \oplus \E_-^r \oplus \E_+^r \oplus
\E_c^r$,
and the projections from $\E$ into these subspaces are defined exactly as those corresponding projections in Theorem \ref{Spectrum.C.R}. Moreover, we have
\[ 
  \bH_r \Phi_{00}^r = \Phi_{01}^r, \ \bH_r \Phi_{0j}^r =0, \
  \bH_r \Phi_k^r = \omega_k \Phi_k^r,  \  \forall \ j = 1,2,3, \
  k = 1,2,3,4,5,6. 
\]
\end{corollary}
In the non-resonant case, we also have the following result 
on the spectrum of $\bH_0$, whose proof is simpler, and is 
therefore skipped:
\begin{theorem}[$\bH_0$-Invariant Subspaces - Non-Resonant Case] 
\label{Spectrum.C.NR} 
Assume that $2e_1 < e_0$. Let $\bH_0$ be defined as in \eqref{bH.def}. Then, the space $\E = L^2(\mathbb{R}^3, \mathbb{C}^2)$ can be decomposed into the $\bH_0$-invariant subspaces as
\[ \E =  \oplus_{j=0}^4 \E_j \oplus \E_{c}. \]
If $f$ and $g$ belong to different subspaces, then $\wei{J f, g} =0$. These subspaces and their corresponding projections satisfy the followings:
\begin{itemize}
\item[\textup{(i)}] 
The subspace $\E_0$ is generated by zero-eigenvectors 
\[
  \Phi_{0j} = \frac{\partial}{\partial r_j} 
  \vt{e^{-ir_1}\wt{Q}_E}(\Gamma_1(r_2,r_3) x)
  \big|_{r_1=r_2=r_3=0}
\] 
for $j=1,2,3$, and a generalized eigenvector 
$\Phi_{00} = \partial_E \vt{\wt{Q}_E}$, with $ \bH_0 \Phi_{01} = \Phi_{00}, \quad \bH_0 \Phi_{0j} = 0, \quad \forall j = 0,2,3$ and moreover,
\[ \wei{J\Phi_{0j}, \Phi_{0k}} = 0, \quad \text{if} \quad (j,k) \notin \{(0,1), (1,0), (2,3), (3,2) \}. \]
The projection $\bP_0(\bH_0) : \E \rightarrow \E_0$ is defined by $\bP_0(\bH_0) = \bP_{01}(\bH_0) + \bP_{02}(\bH_0)$ with
\begin{flalign*}
 \bP_{01}(\bH_0)f & = \beta_1 \wei{J \Phi_{00}, f}\Phi_{01} - \beta_1 \wei{J\Phi_{01}, f}\Phi_{00}, \\ 
 \bP_{02}(\bH_0)f & =  \beta_2 \wei{J \Phi_{02}, f}\Phi_{03} -\beta_2 \wei{J \Phi_{03}, f}\Phi_{02}, \quad \forall \ f \in \E,
\end{flalign*}
where $\beta_1 := \wei{J\Phi_{00}, \Phi_{01}}^{-1} = O(1)$ and $\alpha_2 = \wei{J\Phi_{02}, \Phi_{03}}^{-1} = O(\epsilon^{-2})$.
\item[\textup{(ii)}] 
There exist $\Phi_1 := \begin{bmatrix} -i\phi_1 + \phi_2 \\ \phi_1 + i\phi_2 \end{bmatrix} + O(\epsilon^2), \Phi_2 := \bar{\Phi}_1$ and purely imaginary numbers $\omega_1 := i(\lambda \epsilon^2 + O(\epsilon^3)))$ and $\omega_2 := \bar{\omega}_1$ such that for $j =1,2$, the subspace $\E_j$ is spanned by $\Phi_j$ and 
\[ \bH_0 \Phi_j = \omega_j \Phi_j, \quad \wei{J \Phi_1, \Phi_1} =-4i, \quad \wei{J \Phi_1, \Phi_2} =0. \] 
The projection $\bP_j(\bH_0) : \E \rightarrow \E_j$ is defined by
\[\bP_1(\bH_0) f = -\frac{1}{4i}\wei{J\Phi_1, f}\Phi_1, \ \ \bP_2(\bH_0) f =  \frac{1}{4i}\wei{J \Phi_2, f}\Phi_2, \quad \forall \ f \in \E.\]
\item[\textup{(iii)}] 
There exist $\Phi_3:= \begin{bmatrix} \phi_0 \\ -i\phi_0 \end{bmatrix} + O(\epsilon^2)$, $\Phi_4:= \bar{\Phi}_3$ and purely imaginary numbers $\omega_3 = i(e_1 - e_0 + O(\epsilon^2)), \omega_4 = \bar{\omega}_3$ such that for $j=3,4$ the subspaces $\E_j$ is spanned by $\Phi_j$ and
\[ \bH_0 \Phi_j = \omega_j \Phi_j, \quad \wei{J \Phi_3, \Phi_3} = 2i, \quad \wei{J \Phi_3, \bar{\Phi}_4} =0. \] 
The projection $\bP_j(\bH_0) : \E \rightarrow \E_j$ is defined by
\[\bP_3(\bH_0) f = \frac{1}{2i}\wei{J\Phi_3, f}\Phi_3 \quad \bP_4(\bH_0) f = - \frac{1}{2i}\wei{J \bar{\Phi}_4, f} \bar{\Phi}_4, \quad \forall \ f \in \E.\]
\item[\textup{(iv)}] $\E_c = \{g \in \E : \wei{Jf, g} = 0, \quad \forall \ f \in \E_j, \ \forall \ j = 0,1,2\}$. Its corresponding projection is $\bP_c(\bH_0) = \textup{Id}  - \sum_{j=0}^4 \bP_j(\bH_0)$.
\end{itemize}
\end{theorem}
\medskip

\setlength{\unitlength}{1mm}\noindent
\begin{center}
\begin{picture}(120,80)

\put (0,40){\vector(1,0){120}}
\put (120,42){\makebox(0,0)[c]{\scriptsize x}}

\put (60,0){\vector(0,1){80}}
\put (60,82){\makebox(0,0)[c]{\scriptsize y}}

\multiput(60, 76)(0.5, 0){30}{\circle*{0.1}}\put (75,76){\vector(1,0){1}}
\put(83,76){\makebox(0,0)[c]{\scriptsize$\sigma_c(\bH_0)$}}

{\color{red}\put (60.15,60){\line(0,1){18}}\put (59.8,60){\line(0,1){18}}}

\put(64,60){\makebox(0,0)[c]{\scriptsize$|E|$}}

{\color{red}\put (60.15,20){\line(0,-1){20}}\put (59.8,20){\line(0,-1){20}}}
\put(64,20){\makebox(0,0)[c]{\scriptsize$-|E|$}}

\put(60,40){\circle*{1.5}}

\multiput(60, 40)(0.8, -0.4){30}{\circle*{0.1}}\put (84,28){\vector(2,-1){1}}
\put(100,25){\makebox(0,0)[c]{\scriptsize$\E_0 =\{\Phi_{00} = \partial_EQ_{|r=0}, \Phi_{0j} =\partial_{r_j}Q_{|r=0} , \ j = 1,2,3\}$}}

\put(60,43){\circle*{1.5}}
\put(73,43){\makebox(0,0)[c]{\scriptsize$(\omega_1 =O(\ep^2), \Phi_1)$}}

\put(60,37){\circle*{1.5}}
\put(45,37){\makebox(0,0)[c]{\scriptsize$(\omega_2 =\bar{\omega}_1, \Phi_2 =\bar{\Phi}_1)$}}

\put(60,56){\circle*{1.5}}
\put(81,56){\makebox(0,0)[c]{\scriptsize$(\omega_3 =i(e_1-e_0)+O(\ep^2), \Phi_3)$}}

\put(60,24){\circle*{1.5}}
\put(45,24){\makebox(0,0)[c]{\scriptsize$(\omega_4 =\bar{\omega}_3, \Phi_4 = \bar{\Phi}_3)$}}

\multiput(60, 4)(0.5, 0){30}{\circle*{0.1}}\put (75,4){\vector(1,0){1}}
\put(82,4){\makebox(0,0)[c]{\scriptsize$\sigma_c(\bH_0)$}}

\end{picture}

Figure 4: Spectrum of $\bH_0$ in the non-resonant case.
\end{center}

\begin{corollary} \label{Spectrum.N.R-cor}
For $r=(r_1,r_2,r_3) \in \mathbb{R}^3,$ and for $j=0,1,2,3$, let 
$\Phi_{0j}^r : = r* \Phi_{0j} = e^{-Jr_1} \Phi_{0j} \circ
\Gamma_1(r_2,r_3)$. Similarly, for $j = 1,2,3,4$, let 
$\Phi_j^r := e^{-Jr_1}\Phi_j \circ \Gamma_1(r_2,r_3)$. 
Also, let $\E_j^r$ and $\E_c^r$ be exactly as in 
Theorem \ref{Spectrum.C.NR}, for $j=0,1,2,3,4$.
Then $\E_j^r$ are invariant under $\bH_r$, 
$\E = \oplus_{j=0}^4 \E_j^r \oplus \E_c^r$, 
and the projections from $\E$ into these subspaces are defined exactly as those corresponding projections in Theorem \ref{Spectrum.C.NR}. Moreover,
\[ 
  \bH_\gamma \Phi_{00}^r = \Phi_{01}^r, \ \bH_\gamma \Phi_{0j}^r =0, \
  \bH_\gamma \Phi_k^r = \omega_k \Phi_k^r,  \  \forall \ j = 1,2,3, \
  k = 1,2,3,4. 
\]
\end{corollary}

\subsection{Resolvent Estimates and Decay Estimates}
In this section, we shall study the resolvent 
$R(z) = (\vt{\L}-z)^{-1}$ with $\vt{\L} := \vt{\L}_Q$
for $Q = Q_E$ or $\wt{Q}_E$. Set $e_{01} := e_1 - e_0$.
In the resonant case, let $\omega$ denote the eigenvalue
in the first quadrant. Recall that
$\omega = i \kappa + \gamma$ with $\kappa = e_{01} + O(\epsilon^2)$
and $\gamma = \gamma_0 \epsilon^4 + O(\epsilon^5)$ for 
some constant $\gamma_0 >0$. Also, recall that we can write
\[ 
  \vt{\L} = JH + W, \quad W = \begin{bmatrix} W_1 & W_2 \\ 
  W_3 &  W_4 \end{bmatrix} = O(\epsilon^2), \quad H = H_0 - E. 
\]
\begin{lemma}[Resolvent Estimates] \label{Resol.est} Let $R(z) := (\vt{\L}-z)^{-1}$ be the resolvent of $\vt{\L}$, $\B := \B(\E_{s}, \E_{-s})$ for some $s > 3$ where $\E_s$ is defined in \eqref{Ls}. Then we have a constant $C>0$ independent of $\epsilon$ such that for $\tau \geq |E|$,
\begin{equation} \label{resolvent.est}
\norm{R(i\tau \pm 0)}_{\B}+ \norm{R(-i\tau \pm 0)}_{\B} \leq C [(1+\tau)^{-1/2} + (|\tau - \kappa| + \epsilon^4)^{-1}].
\end{equation}
Moreover, for $k = 1,2$, we also have,
\begin{equation}\label{D-resolvent.est}
 \norm{R^{(k)}(i\tau \pm 0)}_{\B} +\norm{R^{(k)}(-i\tau \pm 0)}_{\B} \leq C [(1+\tau)^{-(1+k)/2} + |\tau - \kappa| + \epsilon^4)^{-1}], 
\end{equation} 
where $R^{(k)}$ is the $k^{\text{th}}$ derivatives of $R$.
\end{lemma}
\begin{proof} We prove the lemma for $\vt{\L} = \bH_r$ in resonant case. For the other cases, the proof is simpler so we skip. Without loss of generality, we may assume $\vt{\L} = \bH_0$. First of all, we study the resolvent operator $R(z)$ for $z \in
\mathbb{C}$ near $\pm i e_{01}$. Assume that $|z - i e_{01}| <
\epsilon$ or $|z + ie_{01}| < \epsilon$. For $\Phi \in \E_{s}, U \in
\E$ such that $(\vt{\L} -z) U = \Phi$, we want to estimate
$\norm{U}_{\E_{-s}}$. We write $U = \Phi^*  + \xi$, where $\Phi^* = a
\Phi^+_0 + b\Phi^- _0$ with $a,b \in \mathbb{C}$ and $\xi \perp
\Phi_0^\pm$. 
Then the equation $(\vt{\L} -z) U = \Phi$ is equivalent to
\begin{equation} \label{sys.re}
\left \{
\begin{array}{ll}
& a[i(e_0-E) +\frac{1}{2}\wei{\Phi_0^+, W\Phi_0^+} -z] + \frac{b}{2}\wei{\Phi_0^+, W\Phi_0^-} + \frac{1}{2}\wei{\Phi_0^+, W\xi} = \frac{1}{2}\wei{\Phi_0^+, \Phi}, \\
& b[-i(e_0-E) +\frac{1}{2}\wei{\Phi_0^-, W\Phi_0^-} -z] + \frac{a}{2}\wei{\Phi_0^-, W\Phi_0^+} + \frac{1}{2}\wei{\Phi_0^-, W\xi} =\frac{1}{2} \wei{\Phi_0^-, \Phi}, \\
& \xi = R_0(z)\bP_c[\Phi - W(\Phi^* +\xi)] + \sum_{j=1}^3\frac{\wei{\Phi_j^{\pm},\Phi - W(\Phi^* +\xi)}}{2[\pm i(e_1-E)-z]}\Phi_j^\pm.
\end{array}
\right .
\end{equation}
Here $R_0(z) = (JH -z)^{-1}$ which is well-defined if $\Re(z) \not=0$. From \eqref{in.com}, we get
\begin{equation}
 \sum_{j=1}^3\frac{\wei{\Phi_j^{\pm},\Phi - W(\Phi^* +\xi)}}{2[\pm i(e_1-E)-z]}\Phi_j^\pm  = \sum_{j=1}^3\frac{\wei{\Phi_j^{\pm},\Phi - W\xi}}{2[\pm i(e_1-E)-z]}\Phi_j^\pm.
\end{equation}
From this and \eqref{sys.re}, we can write $\xi$ as $\xi = \xi_1 + \xi_2 + \xi_3 + \xi_4 + \xi_5$ where
\begin{equation} \label{xi.dec}
\begin{split}
\xi_1 & :  = -a R_0(z)\bP_c W\Phi_0^+ - b R_0(z)\bP_c W\phi_0^-, \\
\xi_2 & := -R_0(z)W\xi_1 - \frac{1}{2[\pm i(e_1-E)-z]}\sum_{j=1}^1\wei{\Phi_j^\pm, W\xi_1}\Phi_j^\pm\\
\xi_3 &= R_0(z)\bP_c \Phi + \frac{1}{2[\pm i(e_1-E)-z]}\sum_{j=1}^1\wei{\Phi_j^\pm, \Phi}\Phi_j^\pm, \\
\xi_4 & := -R_0(z)W\xi_3 - \frac{1}{2[\pm i(e_1-E)-z]}\sum_{j=1}^1\wei{\Phi_j^\pm, W\xi_3}\Phi_j^\pm\\
\xi_5 & := -R_0(z)W(\xi-\xi_1 -\xi_3) \\ 
& \quad \quad \quad - \frac{1}{2[\pm i(e_1-E)-z]}\sum_{j=1}^1\wei{\Phi_j^\pm, W(\xi -\xi_1 -\xi_3)}\Phi_j^\pm \\
\end{split}
\end{equation}
Now, for $z = i\tau + \delta$ with $|\tau - (e_0-e_1)| \leq \ep$ and $0 < \delta \leq \ep^5$. Then, $|\pm i(e_1-E)-z| = O(|e_1-e_0|)$. So, from Lemma \ref{R0.sp.est} and \eqref{xi.dec}, we obtain
\begin{equation}\label{xi1234.est}
\begin{split}
& \norm{\xi_1}_{\E_{-s}} \leq C\ep^2[|a| + |b|], \quad \norm{\xi_2}_{\E_{-s}} \leq C\ep^4[|a| +|b|], \\
& \norm{\xi_3}_{\E_{-s}} \leq C\norm{\Phi}_{\E_{s}}, \quad \norm{\xi_4}_{\E_{-s}} \leq C\ep^2\norm{\Phi}_{\E_{s}}. 
\end{split}
\end{equation}
Similarly, we also have
\[ \norm{\xi_5}_{\E_{-s}} \leq C\ep^2[\norm{\xi_2}_{\E_{-s}} + \norm{\xi_4}_{\E_{-s}} + \norm{\xi_5}_{\E_{-s}}]. \]
So,
\begin{equation} \label{xi5.est}
\norm{\xi_5}_{\E_{-s}} \leq C\ep^4[\norm{\Phi}_{\E_{-s}} + \ep^2(|a| +|b|)].
\end{equation}
From \eqref{xi1234.est} and \eqref{xi5.est}, we obtain 
\begin{equation} \label{xi.re.est}
\norm{\xi}_{\E_{-s}} \leq C [\norm{\Phi}_{\E_s} + \ep^2(|a| +|b|)].
\end{equation}
Recall that 
\begin{equation} \label{tau.0}
 \tau_0 = i(e_0 -E) + \frac{1}{2}\wei{\Phi_0^+, W\Phi_0^+} = i[\kappa + O(\ep^2)], \quad \Re \tau_0 =0. \end{equation}
 Let $M = (M_{ij})_{i,j =1}^{2}$ be the $2\times 2$ matrix defined by
\begin{equation} \label{M.dec}
\begin{split}
M_{11} : & = \tau_0  -\frac{1}{2}\wei{\Phi_0^+, WR_0(z)\bP_cW\Phi_0^+} -z, \\
M_{12} : & = \frac{1}{2}\wei{\Phi_0^+, W\Phi_0^-} - \frac{1}{2}\wei{\Phi_0^+, WR_0(z)\bP_cW\Phi_0^-}, \\
M_{21} : & = \frac{1}{2}\wei{\Phi_0^-, W\Phi_0^+} - \frac{1}{2}\wei{\Phi_0^-, WR_0(z)\bP_cW\Phi_0^+}, \\
M_{22} : & = \bar{\tau}_0  -\frac{1}{2}\wei{\Phi_0^-, WR_0(z)\bP_cW\Phi_0^-} -z.
\end{split}
\end{equation}
Also, let
\begin{equation} \label{XB.def}
X :  = \frac{1}{2}\begin{bmatrix}\wei{\Phi_0^+, \Phi} - \wei{\Phi_0^+, W(\xi_3 +\xi_4)} \\ \wei{\Phi_0^-, \Phi} - \wei{\Phi_0^-, W(\xi_3 +\xi_4)} \end{bmatrix},\quad B  =\frac{1}{2}\begin{bmatrix} \wei{\Phi_0^+, W(\xi_2 +\xi_5)} \\ \wei{\Phi_0^-, W(\xi_2 +\xi_5)} \end{bmatrix}.
\end{equation}
Then, the first two equations of \eqref{sys.re} become
\begin{equation} \label{sys.re1} M \begin{bmatrix} a \\ b \end{bmatrix} + B = X. \end{equation}
From \eqref{xi.re.est} and \eqref{XB.def}, we have
\begin{equation} \label{XB.est}
 |X| \leq C\norm{\Phi}_{\E_s}, \quad |B| \leq C\ep^6[\norm{\Phi}_{E_{s}} + (|a| + |b|)]. \end{equation}
On the other hand, from \eqref{exp.Pperp}, \eqref{R0z.in} and \eqref{tau.0}, we have
\begin{equation}
\begin{split}
M_{11} & = i(\kappa -\tau + O(\ep^2)) - i\la^2 (\phi_0Q^2,(H-iz)^{-1}\bP_c\phi_0Q^2) \\
\ & \quad - 2i\la^2(\phi_0|Q|^2, (H+iz)^{-1}\bP_c\phi_0|Q|^2) - \de,\\
M_{22} & = -i(\kappa + \tau +O(\ep^2)) -\de -\frac{1}{2}\wei{\Phi_0^-, WR_0(z)\bP_cW\Phi_0^-}. 
\end{split}
\end{equation}
Because of $|z - i(e_0 -e_1)| \leq \ep$ and the Fermi Golden Rule, we get
\begin{equation} \label{M.o1}
\begin{split}
& \Im M_{11} = (\kappa -\tau) + O(\ep^2) = O(\ep), \\
& \Re M_{11} = \la^2\Im (\phi_0Q^2,(H-iz)^{-1}\bP_c\phi_0Q^2) + O(\ep^5) \geq \frac{\la^2\la_0\ep^4}{2}, \\
& \Im M_{22} = -(\kappa +\tau) + O(\ep^2) = O(|e_0-e_1|), \\
& \Re M_{22} = O(\ep^4). 
\end{split}
\end{equation}
Note that from the symmetry property of $\wt{Q}$, we get 
\[ (\phi_0^2, (\Re \wt{Q})^2 - (\Im \wt{Q})^2) = O(\ep^4), \quad (\phi_0^2, \Re \wt{Q} \Im \wt{Q}) = 0. \]
So, it follows that
\[ \wei{\Phi_0^+,W\Phi_0^-} = 2i\la(\phi_0^2, \wt{Q}^2), \quad \Re c_1 = 0, \quad \Im c_1 = O(\ep^4). \] 
From this and \eqref{M.dec}, we get
\begin{equation} \label{M.o2}
M_{12} = O(\ep^4), \quad M_{21} = O(\ep^4).
\end{equation} 
From \eqref{M.o1} and \eqref{M.o2}, we can find a constant $C_0>0$ depending on $\la, \la_0$ and $|e_0-e_1|$ such that
\[|\text{det}M| \geq C_0^{-1} [|\kappa -\tau| + \ep^4]. \]
From this, \eqref{sys.re1}, \eqref{XB.est}, we obtain
\begin{flalign*}
|a| + |b| & \leq |M^{-1}(X-B)| \leq C_0[|\kappa -\tau| + \ep^4]^{-1}\{\norm{\Phi}_{\E_s} + \ep^6(|a| + |b|)\}\\
& \leq C_0[|\kappa -\tau| + \ep^4]^{-1}\norm{\Phi}_{\E_s} + C_0\ep^2(|a| + |b|).
\end{flalign*}
Thus, for sufficiently small $\ep$ such that $C_0 \ep^2 < \frac{1}{2}$, we get
\begin{equation} \label{ab.est} |a| + |b| \leq 2C_0[|\tau -\kappa| + \ep^4]^{-1} \norm{\Phi}_{\E_s}. \end{equation}
From \eqref{xi.re.est} and \eqref{ab.est}, we get $\norm{U}_{\E_{-s}}
\leq C[1+ |\tau -\kappa| + \ep^4]^{-1}  \norm{\Phi}_{\E_s}$. 
In other words, we have
\[ 
  \norm{R(i\tau + 0)}_{(\E_{s}, \E_{-s})} \leq C[1+ (|\tau -\kappa| +
  \ep^4)^{-1}], \ \ \forall \  \tau \in \mathbf{R} : |\tau -(e_0-e_1)|
  \leq \ep.
\]
Also, the estimates of $\norm{R(i\tau -0)}_{\B}, \norm{R(-i\tau \pm0)}_{\B}$ can be obtained in a similar way. So, we have proved \eqref{resolvent.est} for $\tau$ near $\pm e_{01}$.

Next, for $z = i\tau +0$ where $\tau > e_{01} +\ep$ or $|E| \leq \tau \leq e_{01} -\ep$, from \cite[Theorem 9.2]{JK}, we have $\norm{R_0(z)}_{\B} \leq C(1+\tau)^{-1/2}$. Then, by a similar argument as we just did, we also obtain \eqref{resolvent.est}.

Finally, to prove \eqref{D-resolvent.est}, we follow the argument in
\cite[Lemma 2.5]{TY3}. Basically, we obtain 
\eqref{D-resolvent.est} by an induction argument and by differentiating the relation $R(z)[1 +WR_0(z)] = R_0(z)$ and using the relations $(1+WR_0)^{-1}=1- WR$, $(1+R_0W)^{-1} = 1 - RW$. The proof of the lemma is then complete.
\end{proof}
\begin{lemma}[Decay Estimates] \label{Decay.est} For any 
of our excited states $Q$, let $\vt{\L} = \vt{\L}_Q$. 
Then we have
\begin{itemize}
\item[\textup{(i)}] 
There exists a constant $C>0$ independent of $\ep$ such that for all $\eta \in \E_c(\vt{\L}) \cap H^k$, we have
\[ C^{-1}\norm{\eta}_{H^k} \leq \norm{e^{t\vt{\L}}\eta}_{H^k} \leq C \norm{\eta}_{H^k}, \ \forall \ \eta \in \E_c(\vt{\L}) \cap H^k, \ \ k=1,2.\]
\item[\textup{(ii)}] For all $p \in [2, \infty]$, there is a constant $C =C(\ep,p)$ such that for all $\eta \in \E_c({\vt{\L}})$, we have
\[ \norm{e^{t\vt{\L}} \eta}_{L^p} \leq  C(\ep) |t|^{-3(1/2-1/p)}\norm{\eta}_{L^{p'}},  \ \quad  \text{where} \quad \frac{1}{p} + \frac{1}{p'} =1. \]
\end{itemize}
\end{lemma}
\begin{proof}
To prove (i), let's define the quadratic form: $\Q[\psi] =
\wei{\vt{\L}\psi, J\psi} = \wei{\bK\psi, \psi}, \ \psi \in \E$.
Recall $\bK$ is self-adjoint.
Then, for all $\psi \in \E$, we have
\begin{flalign*}
 \frac{\partial}{\partial t}\Q[e^{\vt{\L}t}\psi] & =\frac{\partial}{\partial t}\wei{\bK e^{J\bK t}\psi,e^{J\bK t}\psi} \\
 & =\wei{\bK J\bK e^{J\bK t}\psi,e^{J\bK t}\psi} + \wei{\bK e^{J\bK t}\psi,J\bK e^{J\bK t}\psi} =0. 
\end{flalign*}
Therefore, $\Q[e^{\vt{\L}t}\psi] = \mathcal{Q}[\psi]$, for all $t \geq 0$ and
for all $\psi \in \E$. 

Next, we claim that there exists $C>0$ such that 
\begin{equation} 
\label{eq.Hk}
  C^{-1} \norm{\eta}_{H^1}^2 \leq \Q[\eta] \leq C \norm{\eta}_{H^1}^2,
  \ \forall \ \eta \in \E_c(\vt{\L}) \cap H^1.
\end{equation}
The upper-bound is immediate, so we just need the lower bound
(the ``spectral gap'').
Recall $\bK = H - JW$, with $W = O(\e^2)$.
In the non-resonant cases, the lower bound follows, by a simple 
perturbation-theoretic argument, from the corresponding
spectral gap for the reference operator $H$.
However, in the resonant cases, the subspaces $\E_c(\vt{\L})$
and $\E_c(J H)$ are not close, and the argument is more subtle -- 
so let us assume now we are in the resonant case.
Define the subspaces
\[
  S_0 := \text{span} \left\{
  \left[ \begin{array}{c} \phi_j \\ 0 \end{array} \right],
  \left[ \begin{array}{c} 0 \\ \phi_j \end{array} \right]
  \; j = 1,2,3 \right \}, \;\;
  S_1 := \text{span} \left\{
   \Phi_a, \Phi_{b} \right\}, \;\;
  S := S_0 \oplus S_1, 
\]
where for $\vt{\L} = \bL_0$, $(a,b)=(1,3)$,
and for $\vt{\L} = \bH_0$, $(a,b)=(3,4)$
(i.e. one eigenfunction from $\E_+$, one from $\E_-$).
Notice $S$ is $8$-dimensional. For any $v \in S_0$,
\[
  \frac{ \langle v, \bK v \rangle }{ \| v \|^2_{L^2} }
  = e_1-E + O(\e^2) = O(\e^2).
\]
Similarly, if $v \in S_0$, $w \in S_1$,
\[
  \langle v, \bK w \rangle = \langle \bK v, w \rangle
  = O(\e^2) \| v \|_{L^2} \|w\|_{L^2},
\]
and
\[
  \langle v, \; w \rangle = O(\e^2) \| v \|_{L^2} \| w \|_{L^2}
\]
since $\Phi_{a,b} = a_1 \left( \begin{array}{c}
\phi_0 \\ \pm i \phi_0 \end{array} \right) + h$,
for an order-one constant $a_1$, and 
with $\|h\|_{L^\infty} = O(\e^2)$. 
Now, for $\Phi = \Phi_a$ or $\Phi_b$ and
$\omega = \omega_a$ or $\omega_b$: 
\[
\begin{split}
  \bar{\omega} \langle \Phi, J \Phi \rangle &=
  \langle \vt{\L} \Phi, J \Phi \rangle =
  \langle \bK \Phi, \Phi \rangle =
  \langle \Phi, \bK \Phi \rangle \\
  &= -\langle \Phi, J \vt{\L} \Phi \rangle =
  -\omega \langle \Phi, J \Phi \rangle
\end{split}
\]
and since $\bar{\omega} \not= -\omega$,
$\langle \Phi, J \Phi \rangle = 0$, and so 
\[
  \langle \Phi_a, \bK \Phi_a \rangle = 
  \langle \Phi_b, \bK \Phi_b \rangle = 0.
\]
Finally, notice
\[
\begin{split}
  \bar{\omega_a} \langle \Phi_a, J \Phi_b \rangle
  &= \langle \vt{\L} \Phi_a, J \Phi_b \rangle
  = \langle \bK \Phi_a, \Phi_b \rangle
  = \langle \Phi_a, \bK \Phi_b \rangle \\
  &= \langle J \Phi_a, \vt{\L} \Phi_b \rangle
  = \omega_b \langle J \Phi_a, \Phi_b \rangle
  = -\omega_b \langle \Phi_a, J \Phi_b \rangle
\end{split} 
\]
and since $\bar{\omega}_a \not= -\omega_b$,
we have $\langle \Phi_a, J \Phi_b \rangle = 0$ and so
\[
  \langle \Phi_a, \bK \Phi_b \rangle = 0.
\]
Combining all these facts, we conclude that
\[
  v \in S \; \implies \;
  \frac{\langle v, \bK v \rangle}{\| v \|^2_{L^2}} = O(\e^2).
\]
We claim now that:
\begin{equation}
\label{gap}
  J \eta \perp S_0 \oplus \E_+ \oplus \E_- \;\; \implies \;\;
  \langle \eta , \bK \eta \rangle \geq \frac{1}{4} |e_1| 
  \| \eta \|^2_{L^2}.
\end{equation}
Indeed, if not, then for any $v$ in the $9$-dimensional subspace
$S \oplus \langle \eta \rangle$,  we would have 
$\langle v, \bK v \rangle < \frac{1}{2}|e_1| \| v \|^2_{L^2}$.
To see this, we are using
$\langle v, \bK \eta \rangle = O(\e^2) \| v \|_{L^2} \| \eta \|_{L^2}$
if $v \in S$, $J \eta \perp S$, which is easily checked,
as well as
$\| v + \eta \|_{L^2}^2 \sim \| v \|_{L^2}^2 + \| \eta \|_{L^2}^2$
for $v \in S$, $J \eta \perp S_0 \oplus \E_+ \oplus \E_-$,
which, in turn, follows from the non-degeneracy of the 
matrix $\langle \Phi_j, J \Phi_k \rangle$, and a little 
linear algebra.    
The mini-max principle would then imply that $\bK$ has
at least $9$ eigenvalues (counting multiplicity) below 
$\frac{1}{2}|e_1|$, while standard perturbation theory shows,
in fact, that there are just $8$.

Now since the $\Phi_{0j}$ are small $L^2$ perturbations
of elements of $S_0$, we get easily from~\eqref{gap}
\[
  \eta \in \E_c \;\; \implies \;\; 
  \langle \eta, \; \bK \eta \rangle \geq C^{-1} \| \eta \|_{L^2}^2.
\]
Then it is straightforward to upgrade this lower bound to $H^1$
by using $\bK = \delta(-\Delta + |E|) + (1-\delta) \bK + \delta R$,
with $R$ a bounded multiplication operator, for suitably
small $\delta$. This yields~\eqref{eq.Hk}.
From this and since $\Q[e^{\vt{\L}t}\eta] = \Q[\eta]$, we get
\[ 
  \Q[e^{\vt{\L}t}\eta] = \Q[\eta] \sim \norm{\eta}_{H^1}^2, 
\]
which proves (i) for $k=1$. 

A straightforward consequence of the $H^1$ estimate just proved is 
that:
\[ 
  \norm{\eta}_{H^3}^2 \sim \norm{\vt{\L}\eta}_{H^1}^2 \sim
  \Q[\vt{\L}\eta], \ \forall \ \eta \in \E_c \cap H^3. 
\] 
Since $\Q[\vt{\L}\eta] = \Q[e^{t\vt{\L}} \vt{\L}\eta]$, it follows
that $\norm{\eta}_{H^3} \sim \norm{e^{t\vt{\L}}\eta}_{H^3}$. Then, by
interpolation, we obtain $\norm{\eta}_{H^2} \sim
\norm{e^{t\vt{\L}}\eta}_{H^2}$,
which is (i) for $k=2$.

To prove (ii), we follow the argument in \cite[Section 2.7]{TY3}. Let
$H_* = -\Delta -E$, $A: = JV + W$. Then, we have 
$\vt{\L} = JH_* +A$. Note that $H_*$ has no bound states and $A$ 
is localized. Now, we define the wave operator $W_+ =\lim_{t
  \rightarrow \infty} e^{-t\vt{\L}} e^{tJH_*}$. Using Lemma
\ref{Resol.est} and the argument in \cite{C1}, it follows that $W_+$
maps $\E$ onto $\E_c(\vt{\L})$. Moreover, $W_+$ and its inverse
(restricted to 
$\E_c(\vt{\L})$) are bounded $L^p \to L^p$. Then, 
from the intertwining property, we have
\[ 
  e^{t\vt{\L}}\bP_c = W_+ e^{tJH_*} (W_+)^*\bP_c. 
\]
From this and the decay estimates of $e^{tJH_*}$, we obtain (ii).
\end{proof}
\section{Proof of Theorem \ref{m-theorem}} \label{prom}
We shall divide the proof of Theorem \ref{m-theorem} into two cases.
 The first case is for $Q= \tilde{Q}_E$, and the second one is 
for $Q= Q_E$. Each of these cases will be divided into two 
sub-cases depending on whether $e_0 < 2e_1$ (resonant case) or 
$e_0 > 2e_1$ (non-resonant case). We will give the details of the 
proof of Theorem \ref{m-theorem} for the resonant case. In the 
non-resonant case, the proof is similar, and therefore we only 
sketch it.

We draw heavily on~\cite{TY3} in this section.

\subsection{Stable Directions for the Excited State 
$\tilde{Q}_E$}
\subsubsection{The Resonant Case} \label{CR-sub}
For $r = (r_1, r_2, r_3) \in \mathbb{R}^3$, let 
$Z_r := (\Phi_{01}^r, \Phi_{02}^r, \Phi_{03}^r),\ G_r := \Phi_{00}^r$
where $\Phi_{0j}^r,  j =0,1,2,3$ are defined as in 
Corollary \ref{Spectrum.C.R-cor}. Also, let $N_r := (\Phi_1^r, \Phi_2^r, \cdots, \Phi_6^r) \in \E^6$. We shall construct a solution $\psi$ of the equation \eqref{NSE} of the form
\[ 
  \vt{\psi} =  \begin{bmatrix} \Re \psi \\ \Im \psi \end{bmatrix} =
  e^{JEt} [\vt{Q}_{E,r(t)} + h], 
\]
where
\[ h = \begin{bmatrix} h_1 \\ h_2 \end{bmatrix} = k +\eta, \quad k= \begin{bmatrix} k_1 \\ k_2 \end{bmatrix} : = a(t)  G_r + b(t)\cdot N_r, \quad \eta \in \E_c^r\]
with $a \in \mathbb{C}, \ b := (b_1, b_2, \cdots, b_6) \in \mathbb{C}^6$ and
\[ 
  \vt{Q}_{E,r} = e^{Jr_1} \begin{bmatrix} \Re \wt{Q}_E \circ
    \Gamma_1(r_2, r_3) \\ \Im  \wt{Q}_E \circ \Gamma_1(r_2,
    r_3)\end{bmatrix} .
\]
Since \eqref{NSE} is equivalent to
\[ 
  \partial_t \vt{\psi} = (\Delta - V) J \vt{\psi} + \lambda |\vt{\psi}|^2 J \vt{\psi}, \]
we get
\begin{equation} \label{k-CR.eqn}
 \partial_t \vt{Q}_{E,r} + \partial_t h = \bH_r h + F_1, \quad F_1 :=  \lambda [2\vt{Q}_{E,r} \cdot h + |h|^2] Jh + \lambda|h|^2J\vt{Q}_{E,r}. 
\end{equation}
Now, let $\hat{k} : = \nabla_r [a\cdot G_r + b\cdot N_r] = (\hat{k}_1, \hat{k}_2, \hat{k}_3)$, we get
\begin{flalign*}
& \partial_t \vt{Q}_{E,r} = \dot{r} \cdot Z_r, \quad \partial_t h = \dot{r}\cdot \hat{k} + \dot{a}G_r + \dot{b}\cdot N_r +\partial_t\eta, \\
& \bH_r h = a \Phi_{01}^r + \sum_{j=1}^6 \omega_j b_j \Phi_j^r + \bH_{r} \eta.
\end{flalign*}
Therefore, \eqref{k-R.eqn} becomes
\begin{equation} \label{rab.eqn}
\dot{r}\cdot (Z_r + \hat{k}) -a \Phi_{01}^r  + \dot{a} G_r + \dot{b}\cdot N_r + \partial_t\eta = \sum_{j=1}^6 \omega_j b_j \Phi_j^r + \bH_{r} \eta +F_1.
\end{equation}
Taking the inner products of \eqref{rab.eqn} with $J\Phi_{00}^r, J\Phi_{02}^r, J\Phi_{03}^r$, we obtain
\begin{flalign*}
& \wei{J\Phi_{00}, \Phi_{01}} \dot{r}_1 + \sum_{j=1}^3 \dot{r}_j \wei{J\Phi_{00}^r, \partial_{r_j}k} = \wei{J\Phi_{00}^r, F_1} + a\wei{J\Phi_{00}, \Phi_{01}}, \\
& \wei{J\Phi_{03}, \Phi_{02}} \dot{r}_2 + \sum_{j=1}^3 \dot{r}_j \wei{J\Phi_{03}^r, \partial_{r_j}k} = \wei{J\Phi_{03}^r, F_1}, \\
& \wei{J\Phi_{02}, \Phi_{03}} \dot{r}_2 + \sum_{j=1}^3 \dot{r}_j \wei{J\Phi_{02}^r, \partial_{r_j}k} = \wei{J\Phi_{02}^r, F_1}.
\end{flalign*}
In other words, we can write $M \dot{r}^T = A$, where $M = M_0 + M_1$ with 
\begin{equation} \label{M.CR.def}
M_0 : = \begin{bmatrix}\wei{J\Phi_{00}, \Phi_{01}} & 0 & 0 \\ 0 & \wei{J\Phi_{03}, \Phi_{02}} & 0 \\ 0 & 0 & \wei{J\Phi_{02}, \Phi_{03}} \end{bmatrix}, \quad M_1 : = \begin{bmatrix}\wei{J\Phi_{00}^r, \hat{k}} \\ \wei{J\Phi_{03}^r, \hat{k}} \\ \wei{J\Phi_{02}^r, \hat{k}} \end{bmatrix},
\end{equation}
and $A^T : = ( \wei{J\Phi_{00}^r, F_1} + a\wei{J\Phi_{00}, \Phi_{01}}, \wei{J\Phi_{03}^r, F_1}, \wei{J\Phi_{02}^r, F_1})$. We shall show that the matrix $M$ is invertible and therefore, $r$ satisfies $\dot{r} = [M^{-1}A]^T$.  Now, let $F_2 : = \dot{r} \cdot \hat{k} = [M^{-1}A]^T \cdot \hat{k}$ and $F = F_1 -F_2$. Taking the inner product of \eqref{rab.eqn} with $J\Phi_{01}$ and $J\Phi_{j}^r$ for $j =1,2,\cdots, 6$, we obtain
\begin{equation*}
\dot{a}   = \wei{J\Phi_{01}, \Phi_{00}}^{-1}\wei{J\Phi_{01}^r, F}, \quad \dot{b}  = B.
\end{equation*}
Here $B = \omega + B'$ where $\omega = (\omega_1, \omega_2,\cdots, \omega_6)$ and $B' = (B_1', B_2', \cdots, B_6') \in \mathbb{C}^6$ with
\begin{equation} \label{B.CR.def} 
\begin{split}
B'_j & = (J\Phi_j, \Phi_j)^{-1}\wei{J\Phi_j^r,F}, \quad j =1,2.\\
B_3' & = (J\Phi_6,\Phi_3)^{-1}\wei{J\Phi_6^r,F}, \quad B'_6=(J\Phi_3,\Phi_6)^{-1}\wei{J\Phi_3^r,F}, \\
B_4' & = (J\Phi_5,\Phi_4)^{-1}\wei{J\Phi_5^r,F}, \quad B'_5=(J\Phi_4,\Phi_5)^{-1}\wei{J\Phi_4^r,F}.
\end{split}
\end{equation}
Moreover, applying the projection $\bP_c(\bH_r)$ onto \eqref{rab.eqn}, we obtain the equation of $\eta$
\begin{equation} \label{eta.CR.eqn} \partial_t \eta = \bH_r \eta + \bP_c(\bH_r)F. \end{equation}
Since $\bH_r$ is time dependent, it's not easy to estimate $\eta$
directly from \eqref{eta.CR.eqn}. To overcome this difficulty, we
shall use the following transformation: Let 
$\wt{\eta} = \bP_c(\bH_0)\eta$. Then \eqref{eta.CR.eqn} becomes
\begin{equation}
\partial_t \wt{\eta} = \bH_0 \wt{\eta} + \bP_c(\bH_0)NL, \quad NL:=\wt{F} + \bP_c(\bH_r)F,
\end{equation}
where $\wt{F}= (\bH_r - \bH_0)\eta$. Moreover, we have
\begin{flalign} \notag
 & \bP_c(\bH_r) -\bP_c(\bH_0) =\\ \label{U.est}
 & = \sum_{j=0}^2 [\bP_j(\bH_0) -\bP_j(\bH_r)] +[\bP_+(\bH_0) -\bP_+(\bH_r)]  +[\bP_-(\bH_0) -\bP_-(\bH_r)]. 
\end{flalign}
From this, Lemma \ref{L2.eta.sp}, we see that $\bP_c(\bH_r) -\bP_c(\bH_0) = O(|r|)$. Since $\eta = \wt{\eta} + [\bP_c(H_\gamma) -\bP_c(\bH_0)]\eta$, $|r|$ is sufficiently small and $\ep$ is fixed,  we can solve $\eta$ in term of $\wt{\eta}$ by 
\begin{equation} \label{U.CR.def}
\eta = \bU_r \wt{\eta}, \quad \bU_r : = \sum_{j=0}^\infty [\bP_c(\bH_r) -\bP_c(\bH_0)]^j.
\end{equation}
Now, for given $\eta_\infty \in \E_c(\bH_0)$, let $\wt{\eta} = e^{\bH_0t}\eta_\infty  + g$, where $g$ satisfies the equation
\begin{equation*}
\partial_t g = \bH_0 g + \bP_c(\bH_0) NL,
\end{equation*}
and we want $g(t) \rightarrow 0$ in some sense as $t \rightarrow \infty$. In summary, we shall construct a solution $\psi$ of \eqref{NSE} as
\[ 
  \vt{\psi} = e^{JEt} [\vt{Q}_{E,r(t)} + a \Phi_{00}^r + b \cdot
  \Phi_r + \bU_r(\xi + g)], 
\]
where $a, b,r$ satisfy the system of equations
\begin{equation}
\left \{ \begin{array}{ll}
\dot{a}  & = \wei{J\Phi_{01}, \Phi_{00}}^{-1}\wei{J\Phi_{01}^r, F}, \\
\dot{b} & = B, \quad \dot{r} = [M^{-1}A]^T, \\
\dot g & = \bH_0 g + \bP_c(\bH_0)NL.
\end{array} \right.
\end{equation}
Now, for $\delta >0$ and sufficiently small, let
\begin{flalign*}
\mathcal{X} :  = \{ & (a,r,b, g): [0, \infty) \rightarrow  \mathbb{C} \times \mathbb{R}^3 \times \mathbb{C}^6 \times (\E_c(\bH_0) \cap H^2) : \\
\ & |a(t)|, |b(t)|, |r_j(t)| \leq \delta^{7/4} (1+t)^{-2}, \quad j = 2,3;  \quad |r_1(t)| \leq 2\delta^{7/4}(1+t)^{-1}, \\
\ & \norm{g(t)}_{\E \cap H^2} \leq \delta^{7/4}(1+t)^{-3/2}\}.
\end{flalign*}
Then, we define the map $\Omega :\mathcal{X} \rightarrow \mathcal{X}$ with $\Omega(a,r,b,g) = (a^*,r^*,b^*,g^*)$ as
\begin{flalign*}
a^*(t) & = \int_\infty^t [\wei{J\Phi_{01}, \Phi_{00}}^{-1}\wei{J\Phi_{01}^r, F}](s) ds, \\
r^*(t) & = \int_\infty^t [M^{-1}A]^T(s) ds, \\
b^*_j(t) & = \int_\infty^t e^{\omega_j(t-s)} B_j'(s) ds, \quad j = 1,2,\cdots, 4,\\
b^*_k(t) & = e^{\omega_kt}b_{k}(0) + \int_0^t e^{\omega_k(t-s)} B_k'(s) ds, \quad k = 5,6,\\
g^*(t) & = \int_{\infty}^t e^{\bH_0(t-s)} \bP_c NL(s) ds.
\end{flalign*}
Here $F = F_1 - F_2$ where $F_1$ is defined in \eqref{k-CR.eqn} and $F_2 = (M^{-1}B)^T \cdot \hat{k} =(M^{-1}B)^T \cdot \nabla_r k$. Moreover, $\eta = \mathbf{U}_r [e^{\bH_0 t} \eta_\infty + g]$. Note also that for $k =5,6$, we have $\Re \omega_k <0$. Therefore, the terms $e^{\omega_kt}b_{k}(0)$ in the equations of $b^*_k$ exponentially decay and so $b_k(0)$ can be freely chosen. 
\begin{lemma} The map $\Omega$ is well-defined and it is a contraction map if $\delta$ is sufficiently small and 
\[ \norm{\xi_\infty}_{H^2 \cap W^{2,1}} \leq \delta \quad \text{and} \quad |b_k(0)| \leq \delta^2/4, \ \forall \ k =5,6.\]
\end{lemma}
\begin{proof} Recall that $\eta := \xi + \mathbf{U}_r g,\ \xi:=\mathbf{U}_r e^{\bH_0 t} \eta_\infty$. Since $\norm{\xi_\infty}_{H^2 \cap W^{2,1}} \leq \delta$, we get
\[ \norm{\xi(t)}_{H^2} \leq C(\ep)\delta, \quad \norm{\xi(t)}_{W^{2,\infty}} \leq C(\ep) \delta |t|^{-3/2} \]
Therefore,
\begin{equation*}
\norm{|\xi|^2J\xi}_{H^2} \leq C(\ep)\delta^3 (1+t)^{-3}.
\end{equation*}
Now, define $F_0 : = 2(\vt{Q}_{E,r} \cdot \eta) J\eta + |\eta|^2 J\vt{Q}_{E,r} + |\eta|^2 J\eta$. We have $\norm{F_0}_{H^2} \leq \delta^2 (1+t)^{-3}$. From this and \eqref{k-CR.eqn}, we see that $F_0$ is the main term of $F_1$. So, we obtain
\begin{equation} \label{F1.CR.est}
\norm{F_1}_{H^2} \leq C(\ep)\delta^2 (1+t)^{-3}.
\end{equation}
From \eqref{M.CR.def}, we have 
\[ 
  \norm{M_1} \leq C \norm{\hat{k}}_{L^2} \leq C [|a| +|b|] \leq
  C\delta^{7/4}(1+t)^{-2}.  
\]
Therefore, for sufficiently small $\delta$, $M = M_0 + M_1$ is invertible and $M^{-1} = (I_3 + M_0^{-1}M)M_0^{-1}$, where $I_3$ is the identity $3\times 3$ matrix. From this and the explicit formula of $A$, we obtain
\begin{equation}\label{MA.CR.est}
\begin{split}
\norm{(M^{-1}A)_1} &\leq \frac{3}{2}[ |a| + \norm{M_0} \norm{F_1}_{L^2_{-s}}] \leq \frac{3}{2} |a| + C(\ep) \norm{F_1}_{L^2_{-s}}], \\
\norm{(M^{-1}A)_k} & \leq 2\norm{M_0} \norm{F_1}_{L^2_{-s}} \leq C(\ep)\norm{F_1}_{L^2_{-s}}, \ k = 2,3.
\end{split}
\end{equation} 
Then, it follows from \eqref{F1.CR.est} and \eqref{MA.CR.est} that
\begin{equation} \label{F.CR.est}
\norm{F}_{H^2} \leq \norm{F_1}_{H^2} + \norm{F_2}_{H^2} \leq C(\ep) \delta^2 (1+t)^{-3}. \end{equation}
On the other hand, we also have
\[ \norm{\wt{F}}_{H^2} = \norm{(\bH_r - \bH_0)\eta}_{H^2} \leq C\ep |r| \norm{\eta}_{L^\infty} \leq C(\ep) \delta^2(1+t)^{-5/2}. \]
From this and \eqref{F1.CR.est}, we get
\begin{equation} \label{NL.CR.est}
\norm{NL}_{H^2} \leq C(\ep) \delta^2 (1+t)^{-5/2}.
\end{equation}
From Lemma \ref{Decay.est} and \eqref{NL.CR.est}, we obtain
\[
  \norm{g^*(t)}_{H^2} \les C(\ep) \delta^2 \int_{\infty}^t
  \norm{NL(s)}_{H^2} ds \leq C(\ep) \delta^2 (1+t)^{-3/2} \leq
  \delta^{7/4}(1+t)^{-3/2}. 
\]
On the other hand, from \eqref{B.CR.def}, \eqref{MA.CR.est} and \eqref{F.CR.est} and by direct computation, we also obtain
\[ 
  |a^*|, |b^*|, |r_j^*| \leq \delta^{7/4}(1+t)^{-2}, \quad \forall j =
   2,3\quad \text{and} \quad |r_1^*(t)| \leq 2\delta^{7/4}(1+t)^{-1}. 
\]
Hence, $\Omega$ maps $\mathcal{X}$ into $\mathcal{X}$. Also, it is
easily checked that $\Omega$ is a contraction map if $\delta$ is sufficiently small. So, the lemma follows.
\end{proof}
Now, to complete the proof of the Theorem \ref{m-theorem} 
for $Q = \tilde{Q}_E$ in the resonant case, we shall prove that 
\begin{equation} \label{CR.as}
 \norm{\psi_{as}(t) - \psi(t)}_{H^2} \leq C(\ep) (1+t)^{-3/2}. 
\end{equation}
where $\psi, \psi_{as}$ are defined by
\begin{flalign*}
 & \psi  = e^{JEt} [\vt{Q}_{E,r(t)} + a G_r + b \cdot N_r + \bU_r(\xi + g)]\\
 & \psi_{as}(t) = e^{JEt}[\vt{Q}_{E,0} + \xi], \quad  \xi :=e^{\bH_0t}\eta_0. 
\end{flalign*}
We have
\[\psi_{as}(t) - \psi(t) = e^{JEt} [\vt{Q}_{E,0} - \vt{Q}_{E,r} + a G_r + b \cdot N_r + (1-\bU_r) \xi + \bU_r g].  \]
From Lemma \ref{Lp.eta.sp} and \eqref{U.est}, we get
\begin{equation*}
\begin{split}
& \norm{\vt{Q}_{E,0} - \vt{Q}_{E,r} + a G_r + b \cdot N_r}_{H^2} \leq C(\ep)[|r| +|a| +|b|] \\
& \ \quad \quad \leq C(\ep) \delta^{7/4}(1+t)^{-1}, \\
& \norm{\bU_r g}_{H^2} \leq C(\ep) \norm{g}_{H^2} \leq C(\ep) (1+t)^{-3/2}.
\end{split}
\end{equation*}
On the other hand, using \eqref{U.est} again, we obtain
\[\norm{(1-\bU_r)\xi}_{H^2} \leq C(\ep)|r| \norm{\xi}_{H^2} \leq C(\ep) \delta^{11/4}(1+t)^{-3/2}. \]
Therefore, we obtain \eqref{CR.as}.
\subsubsection{The Non-Resonant Case}
We shall construct a solution $\psi$ of the equation \eqref{NSE} of the form
\[ 
  \vt{\psi} =  \begin{bmatrix} \Re \psi \\ \Im \psi \end{bmatrix} =
  e^{JEt} [\vt{Q}_{E,r} + h], 
\]
where
\[ h = \begin{bmatrix} h_1 \\ h_2 \end{bmatrix} = k +\eta, \quad k= \begin{bmatrix} k_1 \\ k_2 \end{bmatrix} : = a(t)  G_r + b(t)\cdot N_r, \quad \eta \in \E_c^r\]
with $a \in \mathbb{C}, \ b := (b_1, b_2, \cdots, b_4) \in \mathbb{C}^4$ and
\[ 
  \vt{Q}_{E,r} = e^{Jr_1} \begin{bmatrix} \Re \wt{Q}_E \circ
  \Gamma_1(r_2, r_3) \\ \Im  \wt{Q}_E \circ \Gamma_1(r_2,
  r_3)\end{bmatrix}.
\]
Here, $G_r = \Phi_{00}^r$ and $N_r = (\Phi_1^r,\cdots, \Phi_4^r)$. As in the previous section, we also obtain the equation of $a,b,\eta$ as
\begin{equation} \label{NR.C.eqn}
\left \{ \begin{array}{ll}
\dot{a}  & = \wei{J\Phi_{01}, \Phi_{00}}^{-1}\wei{J\Phi_{01}^r, F}, \\
\dot{b} & = B, \quad \dot{r} = [M^{-1}A]^T, \\
\dot g & = \bH_0 g + \bP_c(\bH_0)NL.
\end{array} \right.
\end{equation}
where $g = \bU_r[e^{\bH_0t}\eta_\infty +\eta]$ and $F, M, A, NL$ are defined exactly the same way. The map $\Omega$ is defined exactly the same except we don't have the equation for $b^*_5$ and $b_6^*$.
\subsection{Stable Directions for the Excited State 
$Q_E$}
\subsubsection{The Resonant Case} \label{R-R-sub}
For $r = (r_1, r_2, r_3) \in \mathbb{R}^3$, let $Z_r := (\Phi_{01}^r,
\Phi_{02}^r, \Phi_{03}^r), G_r := (\wt{\Phi}_{01}^r, \wt{\Phi}_{02}^r,
\wt{\Phi}_{03}^r) \in \E^3$ where $\Phi_{0j}^r, \wt{\Phi}_{0j}^r, j
=1,2,3$ are defined as in Corollary \ref{Cor.R}. 
Also, let $N_r := (\Phi_1^r, \Phi_2^r, \Phi_3^r, \Phi_4^r) \in \E^4$. 
We shall construct a solution $\psi$ of the equation \eqref{NSE} of the form
\[ 
  \vt{\psi} =  \begin{bmatrix} \Re \psi \\ \Im \psi \end{bmatrix} =
  e^{JEt} [\vt{Q}_{E,r} + h], \quad \text{where} \quad h : = k + \eta,
  \ k:= a(t) \cdot G_r + b(t)\cdot N_r 
\]
with $a := (a_1, a_2, a_3) \in \mathbb{C}^3, \ b := (b_1, b_2, \cdots, b_4) \in \mathbb{C}^4, \eta \in \E_c^r$ and
\[ 
  \vt{Q}_{E,r} = e^{Jr_1} \begin{bmatrix} Q_E \circ 
  \Gamma_0(r_1, r_2) \\ 0\end{bmatrix}.
\]
Since \eqref{NSE} is equivalent to
\[ \partial_t \vt{\psi} = (\Delta -V) J \vt{\psi} + \lambda |\vt{\psi}|^2 J \vt{\psi}. \]
So, we obtain
\begin{equation} \label{k-R.eqn}
 \partial_t \vt{Q}_{E,r} + \partial_t h = \bL_r h + F_1, \quad F_1 :=  \lambda [ 2\vt{Q}_{E,r} \cdot h + |h|^2] Jh + \lambda|h|^2J\vt{Q}_{E,r}. 
\end{equation}
Let $\hat{k} := \nabla_r [a\cdot G_r + b\cdot N_r]$, we have
\begin{flalign*}
& \partial_t \vt{Q}_{E,r} = \dot{r} \cdot Z_r, \\
& \partial_t h = \partial_t \eta + \dot{a}\cdot G_r + \dot{b} \cdot N_r + \dot{r}\cdot \hat{k}, \\
& \bL_r h = \bL_r \eta + a \cdot Z_r + \sum_{j=1}^4 \omega_j b_j \Phi_j^r.
\end{flalign*}
Therefore, the equation \eqref{k-R.eqn} becomes
\begin{equation} \label{R-eqn.all}
(\dot{r} -a)\cdot Z_r + \dot{r}\cdot \hat{k}+ \dot{a}\cdot G_r + \sum_{j=1}^4 (\dot{b}_j - \omega_j b_j)\Phi_j^r  + \partial_t \eta = \bL_r \eta + F_1.
\end{equation}
Now, taking the inner product of \eqref{R-eqn.all} with $\wt{\Phi}^r_{0j}, j =1,2,3$, we obtain
\[ (M_0 +M_1) \dot{r} = M_0 a + A, \]
where
\begin{equation} \label{R-M.def}
M_0 : = \begin{bmatrix}\wei{J\wt{\Phi}_{01}, \Phi_{01}} & 0 & 0 \\ 0 & \wei{J\wt{\Phi}_{02}, \Phi_{02}} & 0 \\ 0 & 0 & \wei{J\wt{\Phi}_{03}, \Phi_{03}} \end{bmatrix}, \quad M_1 : = \begin{bmatrix}\wei{J\wt{\Phi}_{01}^r, \hat{k}} \\ \wei{J\wt{\Phi}_{02}^r, \hat{k}} \\ \wei{J\wt{\Phi}_{03}^r, \hat{k}} \end{bmatrix},
\end{equation}
and $A : = (A_1, A_2, A_3)^T$, where $A_j = \wei{J\wt{\Phi}_{0j}^r, F_1}$. As we will see, the matrix $(M_0 +M_1)$ is invertible. Therefore, we obtain
\begin{equation}\label{R-r.eqn}
\dot{r} = (1 + M_0^{-1}M_1)^{-1}a + (1+M_0^{-1}M_1)M_0^{-1}A.
\end{equation}
Now, for $\eta_\infty \in \E_c(\bL_0)$, let $g$ be such that $\eta = \bU_r[e^{\bL_0 t}\eta_\infty +g]$ where $\bU_r$ is defined exactly the same way as in \eqref{U.CR.def}.  Taking the inner products of \eqref{R-eqn.all} with $J\Phi_{0j}^r, \Phi_k^r$ we also get
\begin{equation} \label{R.ab.eqn} \left \{
\begin{array}{lll}
& \dot{a}_j & =  \wei{J\Phi_{0j}, \wt{\Phi}_{0j}}^{-1} \wei{J\Phi_{0j}^{r}, F}, \\
& \dot{b}_k & = \omega_k b_k + B_k, \\
& \dot{g} & = \bL_0 g + \bP_c(\bL_0) NL,
\end{array} \right.
\end{equation}
where $F= F_1 -F_2, F_2 : = \dot{r} \cdot \hat{k} =[(1 + M_0^{-1}M_1)^{-1}a + (1+M_0^{-1}M_1)M_0^{-1}A]^T \cdot \hat{k}$, $NL = (\bL_r - \bL_0)\eta + \bP_c(\bL_r)F $ and $B$ satisfies $|B| \leq \norm{F}_{L^2_{-s}}$. Now, we define
\begin{flalign*}
\mathcal{Y} :  = \{ & (a,r,b, g): [0, \infty) \rightarrow  \mathbb{C}^3 \times \mathbb{R}^3 \times \mathbb{C}^4 \times (\E_c(\bL_r) \cap H^2) : \\
\ & |a(t)|, |b(t)| \leq \delta^{7/4} (1+t)^{-2}, \quad |r(t)| \leq 2\delta^{7/4}(1+t)^{-1},  \\
\ & \norm{g(t)}_{H^2} \leq \delta^{7/4}(1+t)^{-3/2}\}.
\end{flalign*}
Let $\Omega : \mathcal{Y} \rightarrow \mathcal{Y}$ be a defined as $\Omega(a,r,b,\eta) = (a^*, r^*, b^*, g^*)$ where
\begin{flalign*}
a_j^*(t) &=  \int_{\infty}^t \wei{J\Phi_{0j}, \wt{\Phi}_{0j}}^{-1} \wei{J\Phi_{0j}^{r}, F}(s) ds, \quad j = 1,2,3, \\
r^*(t) & = \int_{\infty}^t [(1 + M_0^{-1}M_1)^{-1}a + (1+M_0^{-1}M_1)M_0^{-1}A](s) ds ,\\
b_k^*(t) & = \int_{\infty}^t e^{\omega_k(t-s)}B_k(s)  ds, \quad k = 1,3,\\
b_l^*(t) & = e^{\omega_lt}b_l(0) + \int_0^t e^{\omega_l(t-s)}B_k(s) ds , \quad l = 2,4, \\
g^*(t) & = \int_{\infty}^t e^{\bL_0(t-s)}\bP_c(\bL_0)NL(s) ds.
\end{flalign*}
Note that $\Re(\omega_l) < 0$ for $l = 2,4$. Therefore, the terms $e^{\omega_l}b_l(0)$ decay exponentially and we only need to require $|b_l(0)| \leq \delta^2/4$. Then, as in Subsection \ref{CR-sub}, we can show that there exist $\ep_0$ and $\delta_0(\ep)$ such that for $0 < \ep \leq \ep_0$ and $0 < \delta \leq \delta_0(\ep)$, the map $\Omega$ is well-defined and contraction. Therefore, we obtain the solution $\psi$ of \eqref{NSE} of the form
\[\psi(t) = e^{JEt}[\vt{Q}_{E,r} + a\cdot G_r + b\cdot N_r + \bU_r(e^{\bL_0}\eta_\infty + g)] \]
and $\psi$ satisfies the Theorem \ref{m-theorem}.
\subsubsection{The Non-Resonant Case}
The construction of the solution $\psi$ is exactly the same as that of subsection \ref{R-R-sub}. The only difference is that $b \in \mathbb{C}^2$, not $\mathbb{C}^4$ as in the Subsection \ref{R-R-sub}. The equation of $b^*$ becomes
\[ b^*(t) = \int_{\infty}^t B(s) ds, \quad B = (B_1, B_2).\]
The equations for $a^*, r^*$ and $g^*$ are the same as those in subsection \ref{R-R-sub}. The proof of Theorem \ref{m-theorem} is now complete.
\section*{Acknowledgments}
We thank T.-P. Tsai and K. Nakanishi for their interest in this work.
S.G. acknowledges the support of NSERC under Grant 251124-07


\bigskip

\noindent{Stephen Gustafson},  gustaf@math.ubc.ca \\
Department of Mathematics, University of British Columbia, 
Vancouver, BC V6T 1Z2, Canada

\bigskip

\noindent{Tuoc Van Phan}, phan@math.ubc.ca \\
Department of Mathematics, University of British Columbia, 
Vancouver, BC V6T 1Z2, Canada

\end{document}